\documentclass[11pt]{article}
\usepackage{amsmath,amssymb,amsthm}
\usepackage{graphicx}
\usepackage{amscd}
\usepackage{amsmath}
\usepackage{caption}
\usepackage{amsfonts}
\usepackage{amssymb}
\usepackage{amsthm}
\usepackage{mathrsfs}
\usepackage{multicol}
\usepackage{color}
\usepackage[english]{babel}
\usepackage[T1]{fontenc}
\usepackage{textcomp}
\usepackage[utf8]{inputenc}
\usepackage{enumitem}
\usepackage{indentfirst}

\newcommand{\RR}{{{\rm I} \kern -.15em {\rm R} }}

\normalbaselines

\newtheorem{Theorem}{Theorem}[section]

\newtheorem{Lemma}[Theorem]{Lemma}

\newtheorem{Corollary}[Theorem]{Corollary}

\newtheorem{Remark}[Theorem]{Remark}

\newcommand{\C}{{{\rm l} \kern -.42em {\rm C} }}

\newcommand{\nat}{{{\rm I} \kern -.15em {\rm N} }}

\newcommand{\be}{\begin{equation}}
\newcommand{\ee}{\end{equation}}
\newcommand{\beq}{\begin{eqnarray}}
\newcommand{\eeq}{\end{eqnarray}}
\newcommand{\beqs}{\begin{eqnarray*}}
\newcommand{\eeqs}{\end{eqnarray*}}
\newcommand{\bt}{\begin{Theorem}}
\newcommand{\et}{\end{Theorem}}
\newcommand{\br}{\begin{Remark}}
\newcommand{\er}{\end{Remark}}
\newcommand{\bc}{\begin{Corollary}}
\newcommand{\ec}{\end{Corollary}}
\newcommand{\bl}{\begin{Lemma}}
\newcommand{\el}{\end{Lemma}}
\newcommand{\bd}{\begin{definition}}
\newcommand{\ed}{\end{definition}}

\renewcommand{\geq}{\geqslant}
\renewcommand{\ge}{\geqslant}
\renewcommand{\leq}{\leqslant}
\renewcommand{\le}{\leqslant}

\title{Energy decay for semilinear evolution equations with memory and time-dependent time delay feedback}
\author{
Elisa Continelli\footnote{Dipartimento di Ingegneria e Scienze dell'Informazione e Matematica, Universit\`{a} di L'Aquila, Via Vetoio, Loc. Coppito, 67010 L'Aquila Italy (\texttt{elisa.continelli@graduate.univaq.it}).}
\and
Cristina Pignotti\footnote{Dipartimento di Ingegneria e Scienze dell'Informazione e Matematica, Universit\`{a} di L'Aquila, Via Vetoio, Loc. Coppito, 67010 L'Aquila Italy (\texttt{cristina.pignotti@univaq.it}).}
}

\begin{document}

\textwidth=160 mm

\textheight=225mm

\parindent=8mm

\frenchspacing

\maketitle

\begin{abstract}
In this paper, we study well-posedness and exponential stability for semilinear second order evolution equations with memory and time-varying delay feedback. The time delay function is assumed to be continuous and bounded. Under a suitable assumption on  the delay feedback, we are able to prove that solutions corresponding to small initial data
are globally defined and satisfy an exponential decay estimate.
\end{abstract}

\vspace{5 mm}

\def\qed{\hbox{\hskip 6pt\vrule width6pt
height7pt
depth1pt  \hskip1pt}\bigskip}

%% {\bf 2000 Mathematics Subject Classification:}
%%35L05, 93D15

 %%{\bf Keywords and Phrases:}  wave equation,  delay feedbacks, stabilization

\section{Introduction}
\label{pbform}
%\hspace{5mm}

\setcounter{equation}{0}

Let $H$ be a Hilbert space and let $A$ be a positive self-adjoint operator with dense domain $D(A)$ in $H.$ Let us consider the system:
\begin{equation}
\label{modello}
\begin{array}{l}
\displaystyle{u_{tt}(t)+Au(t)- \int_0^{+\infty} \beta (s)Au(t-s) ds+k(t)BB^*u_t(t-\tau (t))=\nabla \psi( u(t)), }\\
\hspace{12 cm}\ t\in (0,+\infty),\\
%\hspace{3 cm}
\displaystyle{
 u(t)=u_0(t), \quad t\in (-\infty, 0],}\\
%\hspace{3 cm}\\
\displaystyle{ u_t(t)=g(t), \quad t\in [-\bar\tau,0],}
\end{array}
\end{equation}
where $\bar\tau$ is a fixed positive constant and the function $\tau: [0,+\infty)\rightarrow [0,+\infty)$ represents the time dependent time delay. We assume that the time delay is a continuous function satisfying
\begin{equation}\label{tau_bounded}
\tau(t)\le \bar\tau,\quad  \forall \ t\ge 0.
\end{equation}
In \eqref{modello}, $B$ is a bounded linear operator of $H$ into itself, $B^*$ denotes its adjoint. Also, $(u_0(\cdot), g(\cdot))$ are the initial data taken in suitable spaces and we denote with $u_1:=g(0)$.

Moreover, on the delay damping coefficient  $k:[-\bar\tau,+\infty)\rightarrow \RR$ we assume that   $k(\cdot)\in L^1_{loc} ([-\bar\tau,+\infty))$ and the integral on time intervals of length $\bar\tau$ is uniformly bounded, i.e. there exists a positive constant $K$ such that 
\begin{equation}\label{K}
\int_{t-\bar\tau}^t |k(s)| ds < K, \quad \forall t\geq 0.
\end{equation}
The memory kernel $\beta:[0,+\infty) \rightarrow [0,+\infty)$ satisfies the following classical  assumptions:
\begin{enumerate}[label=(\roman*)]
\item $\beta \in C^1(\RR^+) \cap L^1(\RR^+)$;
\item $\beta(0)=\beta_{0}>0$;
\item $\int_0^{+\infty} \beta(t)dt=\tilde{\beta}<1$;
\item $\beta'(t)\leq -\delta \beta(t)$, for some $\delta>0$.
\end{enumerate}
The nonlinear term satisfies some local Lipschitz continuity property; we will state precisely the assumptions later. 

Time delays are often present in applications and physical models and it is by now well-known that a time lag, even arbitrarily small, may induce instability effects also in systems that are uniformly exponentially stable in the absence of time delays (see e.g. \cite{Datko, NP, XYL}). Nevertheless, suitable choices of the time delay value can return stability (cf. \cite{Gugat, GugatT}) as well as suitable
feedback laws (cf. \cite{NP, XYL}). 

Here, we are interested in studying well--posedness and exponential stability, for small initial data, for the above model \eqref{modello}.
Our results extend the ones in \cite{PP_memory} where the time delay is assumed to be constant (see also \cite{Feng2} for the constant delay case). The extension is not trivial since the classical step by step argument, often used to deal with time delay models, does not work in this case. Moreover, our exponential decay proof, valid also in the constant case, significantly simplifies the scheme of \cite{PP_memory}. 

Time-dependent time delays are also considered in the paper \cite{KP} dealing with abstract evolution equations with delay but, in the case of semilinear wave equations with memory damping,  an extra frictional not delayed damping was needed (see also  \cite{JEE15, JEE18}).  In this paper, we focus then on wave-type equations with viscoelastic damping, delay feedback, and source term,  obtaining well-posedness and stability results for small initial data without adding any extra frictional not delayed dampings.  Moreover, concerning the time delay function, here we work in a very general setting. Indeed, in \cite{KP}, in addition to \eqref{tau_bounded}, it is assumed that the time delay function belongs to
$W^{1,\infty}(0, +\infty)$ with 
$\tau^\prime (t)\le c<1.$ 
These are the standard assumptions used to deal with wave-type equations with  a time dependent time delay (see e.g. \cite{NPV11, ChentoufMansouri, Feng}).

 Here, instead, we only assume that the time delay is a continuous function bounded from above. 
Therefore, our results improve previous related literature in a significant way.

Other models with memory damping and time delay effects have been studied in the recent literature. The first result is due to \cite{KSH}, in the linear setting. In that paper, a standard frictional damping, not delayed, is included in the model to compensate for the destabilizing effect of the delay feedback. As later understood, the viscoelastic damping alone can counter the destabilizing delay effect, under suitable assumptions, without the need of any artificial extra dampings. This has been shown, e.g., in \cite{ANP, Dai, Guesmia, Yang}. The case of intermittent delay feedback has been studied in \cite{P} while the paper \cite{MustafaKafini}
  analyzes a plate equation with memory, source term, delay feedback and, in the same spirit of \cite{KSH}, an extra not delayed frictional damping. Models for wave-type equations with memory damping have been previously studied by several authors in the undelayed case (see e.g.  \cite{Alabau, ACS, CCM}). See also  \cite{A} for results on the Timoshenko model, also in the undelayed case, and extensions to the time delay framework (see e.g. \cite{SS, AM}).

More rich is the literature in the case of frictional/structural damping, instead of a memory term, which compensates for the destabilizing effect of time delays and, for specific models, various stability results have been quite recently obtained under suitable assumptions (see e.g. \cite{AABM, AG, AM, C, CFP, NP, OO, KPFirst, Capistrano, Pignotti24,  XYL}).

The rest of the paper is organized as follows. In Section \ref{Preli}, we detail our functional setting and prove  some preliminary results; in particular, we rewrite system \eqref{modello} in an abstract form. In Section 3, we prove the exponential decay of the energy associated to \eqref{modello}. Finally, in Section 4 some examples are illustrated.

\section{Preliminaries}\label{Preli}

\setcounter{equation}{0}

Here, we precise the functional setting for our analysis and prove some preliminary results.

On the nonlinear term, as in \cite{ACS, PP_memory}, we assume that $\psi : D(A^{\frac 1 2})\rightarrow \RR$ is a functional having G\^{a}teaux derivative $D\psi(u)$ at every $u\in D(A^{\frac 12}).$ Moreover, we assume the following hypotheses:
\begin{itemize}
\item[{(H1)}] For every $u\in D(A^{\frac 1 2})$, there exists a constant $c(u)>0$ such that
$$
|D\psi(u)(v)|\leq c(u) ||v||_{{H}} \qquad \forall v\in {D}(A^{\frac 1 2}).
$$
Then, $\psi$ can be extended to the whole  $H$ and we denote by $\nabla \psi(u)$ the unique vector representing $D\psi(u)$ in the Riesz isomorphism, i.e.
$$
\langle \nabla \psi(u), v \rangle_H =D\psi(u) (v), \qquad \forall v\in H;
$$
\item[ (H2)] for all $r>0$ there exists a constant $L(r)>0$ such that
$$
||\nabla \psi (u)-\nabla \psi (v)||_H \leq L(r) ||A^{\frac 12}(u-v)||_H,
$$
for all $u,v\in {D}(A^{\frac 12})$ satisfying $||A^{\frac 12} u||_H\leq r$ and $||A^{\frac 12} v||_H\leq r$.
\item[{ (H3)}] $\psi(0)=0,$  $\nabla \psi(0)=0$ and
there exists a strictly increasing continuous function $h$ such that
\begin{equation}
\label{stima_h}
||\nabla \psi (u)||_H\leq h(||A^{\frac 12} u||_H)||A^{\frac 12}u||_H,
\end{equation}
for all $u\in {D}(A^{\frac 12})$.
\end{itemize}

As in Dafermos \cite{Dafermos}, we define the function
\begin{equation}
\label{eta}
\eta^t(s):=u(t)-u(t-s),\quad s,t\in (0,+\infty),
\end{equation}
so that we can rewrite \eqref{modello} in the following way:
\begin{equation}
\label{modelloDafermos}
\begin{array}{l}
\displaystyle{u_{tt}(t)+(1-\tilde{\beta})Au(t)+\int_0^{+\infty} \beta (s)A\eta^t(s)ds+k(t)BB^*u_t(t-\tau(t))}\\
\displaystyle{\hspace{8 cm}
=\nabla \psi(u(t)),\ t\in(0,+\infty),}\\
\displaystyle{ \eta^t_t(s)=-\eta^t_s(s)+u_t(t),\ \quad t, s\in  (0, +\infty),}\\
%%\displaystyle{ u(x,t)=0 \quad \text{in} \quad \partial \Omega \times (0,+\infty),}\\
%%\displaystyle{ \eta^t(x,s)=0 \quad \text{in} \quad \partial\Omega \times (0,+\infty),}\\
\displaystyle{ u(0)=u_0(0),}\\
\displaystyle{u_t(t)= g(t), \quad t\in [-\bar\tau, 0],}\\
\displaystyle{ \eta^0(s)=\eta_0(s)=u_0(0)-u_0(-s)  \quad s\in (0,+\infty).}
\end{array}
\end{equation}
Let us define the energy of the model \eqref{modello} (equivalently \eqref{modelloDafermos}) as
\begin{equation}
\label{energia1}
\begin{array}{l}
\vspace{0.3cm}\displaystyle{E(t):=E (u(t))=\frac{1}{2}||u_t(t)||_H^2+\frac{1-\tilde{\beta}}{2}||A^{\frac 12}u(t)||^2_H-\psi(u)  }\\ \hspace{2 cm}
\displaystyle{ +\frac{1}{2}\int_0^{+\infty} \beta(s) ||A^{\frac 1 2} \eta^t(s)||^2_H ds +\frac{1}{2}\int_{t-\bar\tau}^t |k(s)|\cdot ||B^*u_t(s)||_H^2 ds.}
 \end{array}
\end{equation}
Note that, apart from the last term, this is the natural energy for nonlinear wave-type equations with memory (cf. e.g. \cite{ACS}). The additional term
$$\frac 12\int_{t-\bar\tau}^t\vert k(s)\vert\cdot\Vert B^* u_t(s)\Vert_H^2\, ds$$
 is crucial in order to deal with the delay feedback in the case of  time varying time delay  (cf. \cite{KP, PP_memory} for similar terms).

Moreover, let us define the functional 

$$\mathcal{E}(t):=\max\left\{\frac{1}{2}\max_{s\in [-\bar{\tau},0]}\Vert g(s)\Vert_H^2,\,\,\max_{s\in[0,t]}E(s)\right\}.$$ 
In particular, for $t=0$, 

$$\mathcal{E}(0):=\max\left\{\frac{1}{2}\max_{s\in [-\bar{\tau},0]}\Vert g(s)\Vert_H^2,\,\,E(0)\right\}.$$ 

We will show that, thanks to the assumption  \eqref{stima_h}, for solutions to system \eqref{modello} corresponding to sufficiently {\em small} initial data, the energy is positive for any $t\ge 0.$ Moreover, we will prove that an exponential decay estimate holds.

First of all, we will reformulate \eqref{modelloDafermos} (see \eqref{forma_astratta2}) as an abstract first order equation. 
Let us denote
\begin{equation}\label{b}
\Vert B\Vert_{\mathcal{L}(H)}=\Vert B^*\Vert_{\mathcal{L}(H)}=b.
\end{equation}
Our result will be obtained under an assumption on the coefficient $k(t)$ of the delay feedback.
More precisely, we assume (cf. \cite{KP}) that
there exist two constants $\omega '\in [0,\omega)$ and $\gamma\in \RR$ such that
\begin{equation}
\label{ipotesi2}
b^2 Me^{\omega\bar\tau} \int_0^t  |k(s)| ds \leq \gamma +\omega 't, \quad \mbox{\rm for all}\ t\ge 0.
\end{equation}
Note that \eqref{ipotesi2} includes, as particular cases,  $k$ integrable or $k$ in $L^\infty$ with $\Vert k\Vert_\infty$ sufficiently small.

Let $L^2_\beta ((0,+\infty); D(A^{\frac 12}))$ be the Hilbert space of the $D(A^{\frac 12})-$valued functions in $(0,+\infty)$ endowed with the scalar product

$$\langle \varphi, \psi \rangle_{L^2_\beta ((0,+\infty); D(A^{\frac 12}))}=\int_0^\infty \beta (s)\langle A^{\frac 12}\varphi, A^{\frac 12}\psi \rangle_H ds$$
and denote by ${\mathcal H }$ the Hilbert space
$$
\mathcal{H}=D(A^{\frac 12})\times H\times  L^2_{\beta} ((0,+\infty);D(A^{\frac 12})),
$$
equipped with the inner product

\begin{equation}\label{inner}
\left\langle
\left (
\begin{array}{l}
u\\
v\\
w
\end{array}
\right ),
\left (
\begin{array}{l}
\tilde u\\
\tilde v\\
\tilde w
\end{array}
\right )
\right\rangle_{\mathcal{H}}:= (1-\tilde\beta ) \langle A^{\frac 12}u, A^{\frac 12} \tilde u\rangle_H+\langle v, \tilde v\rangle_H+\int_0^\infty \beta (s)\langle  A^{\frac 12}w,  A^{\frac 12} \tilde w\rangle_H ds.
\end{equation}
Setting $U= (u, u_t, \eta^t)$, we can restate \eqref{modello} in the abstract form
\begin{equation}
\label{forma_astratta2}
\begin{array}{l}
\displaystyle{ U'(t)= \mathcal{A} U(t)-k(t)\mathcal{B}U(t-\tau(t))+F(U(t)),}\\
\displaystyle{ U(s)=\tilde g(s), \quad s\in [-\bar\tau, 0],}
\end{array}
\end{equation}
where the operator ${\mathcal A}$  is defined by
\begin{equation*}
\mathcal{A} \begin{pmatrix}
u\\
v\\
w
\end{pmatrix}
=
\begin{pmatrix}
v\\
-(1-\tilde{\beta})Au-\int_0^{+\infty} \beta(s) A w(s) ds \\
-w_s +v
\end{pmatrix}
\end{equation*}
with domain
\begin{equation}
\begin{array}{c}
\displaystyle{ {D}(\mathcal{A})= \{ (u,v,w) \in D(A^{\frac 12})\times D(A^{\frac 12})\times L_{\beta}^2 ((0,+\infty); D(A^{\frac 12})):}\hspace{1,5 cm}\\
\hspace{1.5 cm}
\displaystyle{(1-\tilde{\beta})u+\int_0^{+\infty} \beta (s) w(s)ds \in D(A), \quad w_s \in L^2_{\beta} ((0,+\infty); D(A^{\frac 12}))\},}
\end{array}
\end{equation}
in the Hilbert space ${\mathcal H},$ and the operator ${\mathcal B}:{\mathcal H}\rightarrow {\mathcal H}$ is defined by
$${\mathcal B}\left (
\begin{array}{l}
u\\
v\\
w
\end{array}
\right ):= \left(
\begin{array}{l}
0\\
 BB^* v\\
 0
\end{array}
\right ).$$
Note that, by \eqref{b}, it turns out that $\Vert {\mathcal B}\Vert_{\mathcal{L(H)}} =b^2.$
Moreover,  $\tilde g(s)=(u_0(0), g(s), \eta_0)$ for $s\in [-\bar\tau, 0]$, and we denote $U_0:=\tilde{g}(0)=(u_0(0), u_1, \eta_0)$. Also, $F(U):= (0, \nabla\psi(u), 0)^T.$ From (H2) and (H3) we deduce that
the function $F$ satisfies:
\begin{itemize}
\item[{(F1)}] $F(0)=0$;
\item[{(F2)}] for each $r>0$ there exists a constant $L(r)>0$ such that
\begin{equation}\label{stimaF}
||F(U)-F(V)||_{\mathcal{H}}\leq L(r)||U-V||_{\mathcal{H}}
\end{equation}
whenever $||U||_{\mathcal{H}}\leq r$ and $||V||_{\mathcal{H}}\leq r$.
\end{itemize}
It is well-known (see e.g. \cite{Giorgi}) that the operator ${\mathcal A}$
in the problem's formulation \eqref{forma_astratta2}, corresponding to the linear undelayed part of the model, generates an exponentially stable semigroup
$\{S(t)\}_{t\geq0},$ namely there exist two constants
$M,\omega >0$ such that
\begin{equation}
	\label{semigruppo}
	||S(t)||_{\mathcal{L}(\mathcal{H})}\leq Me^{-\omega t}.
\end{equation}
First of all, we can prove a local well-posedness result.

\begin{Theorem}
	\label{lemma1senzabound}
	Let us consider the system \eqref{forma_astratta2} with initial datum $\tilde g\in C([-\bar{\tau},0]; \mathcal{H}).$ Then, there exists a unique local solution $U(\cdot)$ defined on a time  interval $[0,\delta)$. 
\end{Theorem}
\begin{proof}
	Let $\tilde g\in C([-\bar{\tau},0]; \mathcal{H})$. We set
	$$C:=\max\left\{2M\max_{s\in[-\bar{\tau},0]}\lVert \tilde{g}(s)\rVert_{\mathcal{H}},\max_{s\in[-\bar{\tau},0]}\lVert \tilde{g}(s)\rVert_{\mathcal{H}}\right\}.$$ 
	Let $\xi >0$ be a sufficiently small time such that
	\begin{equation}\label{smalltime}
		\xi L(C)+b^2\lVert k\rVert_{\mathcal{L}^1([0,\xi];\RR)}<\frac{1}{4M},
	\end{equation}
with $L(C)$ as in \eqref{stimaF}.
	Let us denote 
	$$C_{\tilde{g}}([-\bar\tau,\xi];{\mathcal H}):=\{U\in C([-\bar\tau,\xi];{\mathcal H}):U(s)=\tilde{g}(s),\, \forall s\in [-\bar\tau,0]\}.$$ Note that $ C_{\tilde{g}}([-\bar\tau,\xi];{\mathcal H})$ is a nonempty and closed subset of $C([-\bar\tau,\xi];{\mathcal H})$.
	As a consequence, $( C_{\tilde{g}}([-\bar\tau,\xi];{\mathcal H}),\lVert \cdot\rVert_{C([-\bar\tau,\xi];{\mathcal H})})$ is a Banach space.
	Moreover, let us denote 

$$C^C_{\tilde{g}}([-\bar\tau,\xi];{\mathcal H}):=\{U\in C_{\tilde{g}}([-\bar\tau,\xi];{\mathcal H}):\lVert U(t)\rVert_{\mathcal{H}}\leq C,\, \forall t\in [-\bar\tau,\xi]\}.$$

	Let us note that $C^C_{\tilde{g}}([-\bar\tau,\xi];{\mathcal H})$ is nonempty since it suffices to take 
	$$U(s)=\begin{cases}
		U_0,\quad &s\in [0,\xi],\\
		\tilde{g}(s), \quad& s\in [-\bar{\tau},0),
	\end{cases}$$
to have that $U\in C_{\tilde{g}}([-\bar\tau,\xi];{\mathcal H})$ and $\lVert U(t)\rVert_{\mathcal{H}}\leq \lVert \tilde{g}\rVert_{C([-\bar\tau,\xi];{\mathcal H})}\leq C$, for all $t\in [-\bar\tau,\xi]$. So, $U$ belongs to $C^C_{\tilde{g}}([-\bar\tau,\xi];{\mathcal H})$. Also, it is easy to see that $C^C_{\tilde{g}}([-\bar\tau,\xi];{\mathcal H})$ is closed in $C([-\bar\tau,\xi];{\mathcal H})$. Hence, $( C^C_{\tilde{g}}([-\bar\tau,\xi];{\mathcal H}),\lVert \cdot\rVert_{C([-\bar\tau,\xi];{\mathcal H})})$ is a Banach space too.
\\Next, we define the map
	$\Gamma:C^C_{\tilde{g}}([-\bar\tau,\xi];{\mathcal H})\rightarrow C^C_{\tilde{g}}([-\bar\tau,\xi];{\mathcal H})$ given by
	$$\Gamma U(t)=\begin{cases}
		S(t)U_0+\int_{0}^t S(t-s)[F(U(s))+k(s)\mathcal{B}U(s-\tau(s))]\, ds,\quad &t\in (0,\xi],
		\\\tilde{g}(t),\quad &t\in [-\bar{\tau},0].
	\end{cases}$$
We claim that $\Gamma$ is well-defined. Indeed, let $U\in C^C_{\tilde{g}}([-\bar{\tau},\xi];H)$. Then, from the semigroup theory, $t\mapsto S(t)U_0$ is continuous. Also, since $U(\cdot)$ is continuous in $[-\bar{\tau},\xi]$, $\tau(\cdot)$ is a continuous function and ${\mathcal B}$ is a bounded linear operator from $H$ into itself, $[0,\xi]\ni t\mapsto {\mathcal B}U(t-\tau(t))$ is continuous. Moreover, $k\in {\mathcal L}^1([0,\xi];\RR )$. So $k(\cdot){\mathcal B}U(\cdot-\tau(\cdot))\in {\mathcal L}^1([0,\xi];\mathcal{H}  )$. Also, since $U(\cdot)$ is continuous in $[-\bar{\tau},\xi]$ and $F(\cdot)$ is locally Lipschitz continuous in $\mathcal{H}$ from $(F_2),$ the map $t\to F(U(t))$ is continuous in $[-\bar{\tau},\xi]$. So, $F((U(\cdot)))\in {\mathcal L}^1([0,\xi];\mathcal{H} )$. As a consequence, the map $t\mapsto \int_0^t S(t-s)[F(U(s))+ k(s){\mathcal B}U(s-\tau(s))]ds$ is continuous in $[0,\xi]$. Thus, $\Gamma U\in C([0,\xi];H)$. Furthermore, $\Gamma U=\tilde{g}$ in $[\bar{\tau},0]$. Finally, for all $t\in [-\bar{\tau},0]$, 
$$\lVert \Gamma U(t)\rVert_{\mathcal{H}}=\lVert \tilde{g}(t)\rVert_{\mathcal{H}}\leq C.$$
On the other hand, for all $t\in (0,\xi]$, from $(F_2)$ with $F(0)=0$ and $\lVert U\rVert_{C([-\bar\tau,\xi];{\mathcal H})}\leq C$, we can write
$$\begin{array}{l}
	\displaystyle{\lVert\Gamma U(t)\rVert_{\mathcal{H}}\leq Me^{-\omega t}\lVert U_0\rVert_{\mathcal{H}}+M\int_{0}^{t}e^{-\omega(t-s)}(\lVert F(U(s))\rVert_{\mathcal{H}}+b^2\lvert k(s)\rvert\lVert U(s-\tau(s))\rVert_{\mathcal{H}}) ds}\\
	\displaystyle{\hspace{1.8cm}\leq M\lVert U_0\rVert_{\mathcal{H}}+M\int_{0}^{t}(L(C)+b^2\lvert k(s)\rvert)(\lVert U(s)\rVert_{\mathcal{H}}+\lVert U(s-\tau(s))\rVert_{\mathcal{H}}) ds}\\
	\displaystyle{\hspace{1.8cm}\leq M\lVert U_0\rVert_{\mathcal{H}}+2MC(\xi L(C)+b^2\lVert k\rVert_{\mathcal{L}^1([0,\xi];\RR)}).}
\end{array}$$
Thus, using \eqref{smalltime}, by definition of $C$ we get
$$\begin{array}{l}
	\displaystyle{\lVert\Gamma U(t)\rVert_{\mathcal{H}}\leq \frac{C}{2}+2MC(\xi L(C)+b^2\lVert k\rVert_{\mathcal{L}^1([0,\xi];\RR)})\leq \frac{C}{2}+2MC\frac{1}{4M}=C.}
\end{array}$$
Thus, $$\lVert\Gamma U(t)\rVert_{\mathcal{H}}\leq C,\quad\forall t\in [-\bar{\tau},\xi].$$
So, we can conclude that $\Gamma$ is well defined.

Next, we claim that $\Gamma$ is a contraction. Indeed, let $U,V\in C^C_{\tilde{g}}([-\bar{\tau},\xi];H)$. Then, for all $t\in [-\bar{\tau},0]$,
$$\lVert \Gamma U(t)-\Gamma V(t)\rVert_{\mathcal{H}}=0.$$
On the other hand, for all $t\in (0,\xi]$, since $\lVert U\rVert_{C([-\bar\tau,\xi];{\mathcal H})},\lVert V\rVert_{C([-\bar\tau,\xi];{\mathcal H})}\leq C$, from $(F_2)$ it follows that
	$$\begin{array}{l}
		\vspace{0.3cm}\displaystyle{\lVert \Gamma U(t)-\Gamma V(t)\rVert_{\mathcal{H}}\leq \int_{0}^{t}\lVert S(t-s)\rVert_{\mathcal{L}(\mathcal{H})} \lVert F(U(s))-F(V(s))\rVert_{\mathcal{H}} ds}\\
		\vspace{0.3cm}\displaystyle{\hspace{2cm}+ \int_{0}^{t}\lVert S(t-s)\rVert_{\mathcal{L}(\mathcal{H})}\lvert k(s)\rvert \lVert \mathcal{B}U(s-\tau(s))-\mathcal{B}V(s-\tau(s))\rVert_{\mathcal{H}} ds}\\
		
		\displaystyle{\hspace{1.5cm}\leq M(L(C)\xi+b^2 \lVert k\rVert_{\mathcal{L}^1([0,\xi];\RR)})\lVert U-V\rVert_{C([-\bar{\tau},\xi];{\mathcal H})}}
	\end{array}$$
	Thus, 
	$$\lVert \Gamma U-\Gamma V\rVert_{C([-\bar{\tau},\xi];{\mathcal H})}\leq M(L(C)\xi+b^2\lVert k\rVert_{\mathcal{L}^1([0,\xi];\RR)})\lVert U-V\rVert_{C([-\bar{\tau},\xi];{\mathcal H})}.$$
	As a consequence, since from \eqref{smalltime} $M(L(C)\xi+b^2\lVert k\rVert_{\mathcal{L}^1([0,T];\RR)})<\frac{1}{4}<1,$ the map $\Gamma$ is a contraction. Thus, from the Banach's Theorem, $\Gamma$ has a unique fixed point $U\in C^C_{\tilde{g}}([-\bar{\tau},\xi];{\mathcal H})$. So, the fixed point $U\in C^C_{\tilde{g}}([-\bar{\tau},\xi];{\mathcal H})$ is a local solution to \eqref{modello2} that can be extended to some maximal interval $[0,\delta)$ since $\lVert U\rVert_{C([0,\xi];\mathcal{H})}\leq C$. 
	
Now, we prove that the fixed point $U$ is the unique local mild solution to \eqref{forma_astratta2}. Indeed, assume that \eqref{forma_astratta2} has another local mild solution $V$ defined in a time interval $[0,\delta')$. Let $t_0>0$ be such that both $U$ and $V$ are defined in the time interval $[0,t_0]$. We denote with $c:=\max\{\lVert U\rVert_{C([-\bar\tau,t_0];{\mathcal H})},\lVert V\rVert_{C([-\bar\tau,t_0];{\mathcal H})}\}$. Then, for every $t\in [0,t_0]$, we have that
	$$\begin{array}{l}
		\displaystyle{\lVert U(t)-V(t)\rVert_{\mathcal{H}}\leq M\int_{0}^{t}b^2\lvert k(s)\rvert \lVert U(s-\tau(s))-V(s-\tau(s))\rVert_{\mathcal{H}}ds}\\
		\displaystyle{\hspace{3cm}+ML(c)\int_{0}^{t}\lVert U(s)-V(s)\rVert_{\mathcal{H}}ds}\\
		\displaystyle{\hspace{2cm}\leq M\int_{0}^{t}(b^2\lvert k(s)\rvert+L(c))\max_{r\in[s-\bar{\tau},s]}\lVert U(r)-V(r)\rVert_{\mathcal{H}}ds,}
	\end{array}$$
from which 
$$\begin{array}{l}
	\displaystyle{\max_{r\in[t-\bar{\tau},t]}\lVert U(r)-V(r)\rVert_{\mathcal{H}}\leq M\int_{0}^{t}(b^2\lvert k(s)\rvert+L(c))\max_{r\in[s-\bar{\tau},s]}\lVert U(r)-V(r)\rVert_{\mathcal{H}}ds.}
\end{array}$$
Thus, the Gronwall's estimate yields
$$\max_{r\in[t-\bar{\tau},t]}\lVert U(r)-V(r)\rVert_{\mathcal{H}}\leq 0,$$
and $$\lVert U(t)-V(t)\rVert_{\mathcal{H}}=0,\quad \forall t\in [0,t_0].$$
So, $U$ and $V$ coincide on every closed interval $[0,t_0]$ in which they both exist. Then, $\delta=\delta'$ and $U$ is the unique local mild solution to \eqref{modello2}.
\end{proof}
\begin{Remark}
	Assume that the time delay function $\tau(\cdot)$ is bounded from below by a positive constant, namely 
	\begin{equation}\label{lowerbound}
		\tau(t)\geq \tau_0,\quad\forall t\geq 0,
	\end{equation}
	for some $\tau_0>0$. In this case, Theorem \ref{lemma1senzabound} can be proved in a simpler way. Indeed, in $[0,\tau_0]$, we can rewrite the abstract system \eqref{forma_astratta2} as an undelayed problem:
	\begin{eqnarray*}
		U'(t)&=& \mathcal{A}U(t)-k(t)\mathcal{B}\tilde{g}(t-\tau(t))+F(U(t)), \quad t\in (0, \tau_0),\\
		U(0)&=&U_0.
	\end{eqnarray*}
	Then, we can apply the classical theory of nonlinear semigroups (see e.g. \cite{Pazy, PataZucchi}) obtaining the existence of a unique solution  on a set $[0,\delta)$, with $\delta \le\tau_0$.
\end{Remark}
In order to prove our global well-posedness and stability result, we need some preliminary estimates.
\begin{Lemma}\label{stimaE}
	Let $u:[0,T)\rightarrow \RR$ be a solution of \eqref{modello}.
	Assume that
	\begin{equation} \label{e(t)}
		E(t)\geq \frac{1}{4}||u_t(t)||_H^2,\quad \forall t\geq 0.
	\end{equation}
	Then,
	\begin{equation}
		\label{disuguaglianza energia}
		E(t)\leq \bar{C}(t)\mathcal{E}(0),\quad \forall t\geq 0,
	\end{equation}
	where
	\begin{equation}\label{Cbar}
		\bar{C}(t)=e^{3b^2\int_0^t|k(s)|  ds}.
	\end{equation}
\end{Lemma}
\begin{proof}
	Differentiating the energy, we obtain
	$$
	\begin{array}{l}
		\vspace{0.3cm}
		\displaystyle{\frac{dE(t)}{dt}=\langle u_t(t),u_{tt}(t)\rangle_H+(1-\tilde{\beta})\langle A^{\frac 1 2}u(t),A^{\frac 1 2} u_t(t) \rangle_H -\langle \nabla \psi (u(t)), u_t(t)\rangle_H  }\\
		\vspace{0.3cm}
		\displaystyle{\hspace{1.8 cm}+\frac{1}{2} |k(t)|\cdot ||B^*u_t(t)||_H^2-\frac{1}{2}|k(t-\bar{\tau})|\cdot ||B^*u_t(t-\bar\tau)||_H^2}\\
		\displaystyle{\hspace{1.8 cm}+\int_0^{+\infty} \beta (s) \langle A^{\frac 12}\eta^t(s),A^{\frac 1 2}\eta^t_t(s) \rangle_H ds.}
	\end{array}
	$$
	Then, since from \eqref{modelloDafermos} it holds that
	$$u_{tt}(t)=\nabla\psi(u(t))-(1-\tilde{\beta})Au(t)-\int_{0}^{+\infty}\beta(s)A\eta^t(s)ds-k(t)BB^*u_t(t-\tau(t)),$$
	we get
	$$
	\begin{array}{l}
		\vspace{0.3cm}\displaystyle{\frac{dE(t)}{dt}=\langle u_t(t),\nabla\psi(u(t))\rangle_H-(1-\tilde{\beta})\langle u_t(t),Au(t)\rangle_H-\int_{0}^{+\infty}\beta(s)\langle u_t(t),A\eta^t(s)\rangle_Hds}\\
		\vspace{0.3cm}\displaystyle{\hspace{2 cm}-k(t)\langle u_t(t),BB^*u_t(t-\tau(t))\rangle_H+(1-\tilde{\beta})\langle A^{\frac 1 2}u(t),A^{\frac 1 2} u_t(t) \rangle_H -\langle \nabla \psi (u(t)), u_t(t)\rangle_H }\\
		\vspace{0.3cm}\displaystyle{\hspace{2cm}+\frac{1}{2} |k(t)|\cdot ||B^*u_t(t)||_H^2-\frac{1}{2}|k(t-\bar{\tau})|\cdot ||B^*u_t(t-\bar\tau)||_H^2}\\
		\displaystyle{\hspace{2cm}+\int_0^{+\infty} \beta (s) \langle A^{\frac 12}\eta^t(s),A^{\frac 1 2}\eta^t_t(s) \rangle_H ds.}
	\end{array}
	$$
	Let us note that, being $A$ a self-adjoint positive operator, also $A^{\frac{1}{2}}$ is self-adjoint. This together with the second inequality in \eqref{modelloDafermos}, i.e. $\eta_t^t=-\eta_s^t+u_t$, yields
		$$
	\begin{array}{l}
		\vspace{0.3cm}\displaystyle{\frac{dE(t)}{dt}=-\int_{0}^{+\infty}\beta(s)\langle u_t(t),A\eta^t(s)\rangle_Hds-k(t)\langle u_t(t),BB^*u_t(t-\tau(t))\rangle_H}\\
		\vspace{0.3cm}\displaystyle{\hspace{2cm}+\frac{1}{2} |k(t)|\cdot ||B^*u_t(t)||_H^2-\frac{1}{2}|k(t-\bar{\tau})|\cdot ||B^*u_t(t-\bar\tau)||_H^2}\\
		\vspace{0.3cm}\displaystyle{\hspace{2cm}+\int_0^{+\infty} \beta (s) \langle A\eta^t(s),\eta^t_t(s) \rangle_H ds}\\
		\vspace{0.3cm}\displaystyle{\hspace{1.5cm}=-k(t)\langle u_t(t),BB^*u_t(t-\tau(t))\rangle_H+\frac{1}{2} |k(t)|\cdot ||B^*u_t(t)||_H^2}\\
		\displaystyle{\hspace{2cm}-\frac{1}{2}|k(t-\bar{\tau})|\cdot ||B^*u_t(t-\bar\tau)||_H^2-\int_0^{+\infty} \beta (s) \langle A\eta^t(s),\eta^t_s(s) \rangle_H ds.}
	\end{array}
	$$
	Now, we claim that
	\begin{equation}\label{nonnegative}
		\int_0^{+\infty} \beta(s) \langle \eta ^t_s, A\eta^t(s) \rangle_H ds \ge 0.
	\end{equation}
	Indeed, since $A^{\frac{1}{2}}$ is self-adjoint, we can write
	$$\frac{1}{2}\frac{d}{ds}||A^{\frac{1}{2}} \eta^t(s)||_H^2=\langle A^{\frac{1}{2}}\eta ^t_s, A^{\frac{1}{2}}\eta^t(s) \rangle_H=\langle \eta ^t_s, A\eta^t(s) \rangle_H.$$
	Thus, since $\eta^t(0)=0$ and $\beta(t)||A^{\frac{1}{2}} \eta^t(s)||_H^2\to 0$, as $t\to +\infty$, (see \cite{Pata} for details) it comes that
	$$
	\int_0^{+\infty} \beta(s) \langle \eta ^t_s, A\eta^t(s) \rangle_H ds = -\frac{1}{2}\int_0^{+\infty} \beta'(s) ||A^{\frac{1}{2}} \eta^t(s)||_H^2 ds.$$
	Finally, using again (iv) on the memory kernel $\beta(\cdot)$, we can say that
	$$
	\begin{array}{l}
		\vspace{0.3cm}\displaystyle{\int_0^{+\infty} \beta(s) \langle \eta ^t_s, A\eta^t(s) \rangle_H ds = -\frac{1}{2}\int_0^{+\infty} \beta'(s) ||A^{\frac{1}{2}} \eta^t(s)||_H^2 ds}\\
		\displaystyle{\hspace{2cm}\geq \frac{\delta}{2} \int_0^{+\infty} \beta(s) ||A^{\frac{1}{2}} \eta^t(s)||_H^2 ds\geq 0,}
	\end{array},$$
	which proves \eqref{nonnegative}.
	\\Next, from \eqref{nonnegative}, we can estimate the derivative of the energy in the following way:
	$$
	\begin{array}{l}
		\vspace{0.3cm}\displaystyle{\frac{dE(t)}{dt} \leq -k(t)\langle u_t(t),BB^*u_t(t-\tau(t))\rangle_H +\frac{1}{2}|k(t)| \cdot ||B^*u_t(t)||_H^2-\frac{1}{2}|k(t-\bar{\tau})|\cdot ||B^*u_t(t-\bar\tau)||_H^2}\\
		\displaystyle{\hspace{1.1cm}\leq -k(t)\langle u_t(t),BB^*u_t(t-\tau(t))\rangle_H +\frac{1}{2}|k(t)| \cdot ||B^*u_t(t)||_H^2.}
	\end{array}
	$$
	Therefore, using the definition of adjoint and Young inequality, we get
	$$
	\begin{array}{l}
		\vspace{0.3cm}\displaystyle{
			\frac{dE(t)}{dt}\leq-k(t)\langle B^*u_t(t),B^*u_t(t-\tau(t))\rangle_H +\frac{1}{2}|k(t)| \cdot ||B^*u_t(t)||_H^2}\\
		\vspace{0.3cm}\displaystyle{\hspace{2 cm}\leq \frac{1}{2}|k(t)\vert\cdot||B^*u_t(t)||_H^2 +\frac{1}{2}|k(t)\vert\cdot||B^*u_t(t-\tau(t))||_H^2+\frac{1}{2}|k(t)\vert\cdot||B^*u_t(t)||_H^2}\\
		\displaystyle{\hspace{2cm}\leq \frac{3}{2}|k(t)\vert \max_{s\in [t-\bar{\tau},t]}||B^*u_t(s)||_H^2.}
	\end{array}
	$$
	Now, let us note that, from \eqref{e(t)}, for $t\geq \bar{\tau}$ it holds that
	$$\max_{s\in [t-\bar{\tau},t]}\{||B^*u_t(s)||_H^2\}\leq\max_{s\in [0,t]}\{||B^*u_t(s)||_H^2\}\leq b^2\max_{s\in [0,t]}\{||u_t(s)||_H^2\} \leq  2b^2\max_{s\in [0,t]}E(s)\leq 2 b^2 \mathcal{E}(t).$$
	On the other hand, if $t\in [0,\bar{\tau})$, using again \eqref{e(t)}, or 
	$$\max_{s\in [t-\bar{\tau},t]}\{||B^*u_t(s)||_H^2\}=\max_{s\in [0,t]}\{||B^*u_t(s)||_H^2\}\leq 2 b^2 \mathcal{E}(t),$$
	or 
	$$\max_{s\in [t-\bar{\tau},t]}\{||B^*u_t(s)||_H^2\}=\max_{s\in [-\bar{\tau},0]}\{||B^*u_t(s)||_H^2\}\leq b^2\max_{s\in [-\bar{\tau},0]}\{||g(s)||_H^2\}\leq 2 b^2 \mathcal{E}(t).$$
	Therefore, 
	$$\max_{s\in [t-\bar{\tau},t]}\{||B^*u_t(s)||_H^2\}\leq 2b^2 \mathcal{E}(t),\quad \forall t\geq 0,$$
	from which $$\frac{dE(t)}{dt}\leq 3 b^2|k(t)\vert\mathcal{E}(t),\quad \forall t\geq 0.$$
	As a consequence, since $\mathcal{E}(t)$ is constant or increases like $E(t)$, it turns out that
	$$\frac{d\mathcal{E}(t)}{dt}\leq 3b^2|k(t)\vert\mathcal{E}(t),\quad \forall t\geq 0.$$
	Then, the Gronwall's inequality yields
	$$\mathcal{E}(t)\leq e^{3b^2\int_{0}^{t}|k(s)\vert ds}\mathcal{E}(0).$$
	By definition of $\mathcal{E}(t)$, we finally get
	$$E(t)\leq \mathcal{E}(t)\leq e^{3b^2\int_{0}^{t}|k(s)\vert ds}\mathcal{E}(0),$$
	from which
	$$E(t)\leq  e^{3b^2\int_{0}^{t}|k(s)\vert ds}\mathcal{E}(0),$$
that ends the proof. \end{proof}

Before proving our well-posedness and stability results, we need some preliminary estimates. 

\begin{Lemma}
	\label{Lemma 2}
	Let $U(t)=(u(t),u_t(t),\eta^t)$ be a non-zero solution to \eqref{forma_astratta2} defined on an interval $[0, \delta),$ and let $T>\delta.$ Let $h$ be the strictly increasing function appearing in \eqref{stima_h}. 
	\begin{enumerate}
		\item If $h(||A^\frac{1}{2} u_0(0)||_H)<\frac{1-\tilde{\beta}}{2}$, then $E(0)>0$.
		\item Assume that $h(||A^\frac{1}{2} u_0(0)||_H)<\frac{1-\tilde{\beta}}{2}$ and that
		\begin{equation}\label{h1}
			h \left( \frac{2}{(1-\tilde{\beta})^\frac{1}{2}} C^{\frac{1}{2}}\mathcal{E}^\frac{1}{2}(0) \right) <\frac{1-\tilde{\beta}}{2},
		\end{equation}
		for some positive constant $C\geq \bar{C}(T)$, with $\bar{C}(\cdot)$ defined in  \eqref{Cbar}. Then
		\begin{equation}
			\label{stima E dal basso}
			\begin{array}{l}
				\vspace{0.3cm}\displaystyle{ E(t)>\frac{1}{4}||u_t(t)||_H^2+\frac{1-\tilde{\beta}}{4}||A^\frac{1}{2}u(t)||_H^2+\frac{1}{4}\int_{t-\bar\tau}^t |k(s)| \cdot ||B^*u_t(s)||_H^2 ds}\\
				\hspace{1.5 cm}
				\displaystyle{ +\frac{1}{4}\int_0^{+\infty} \beta(s) ||A^\frac{1}{2}\eta^t (s)||_H^2 ds,}
			\end{array}
		\end{equation}
		for all $t\in[0, \delta)$. In particular,
		\begin{equation}\label{J2}
			E(t)>  \frac 14 \Vert U(t)\Vert_{\mathcal H}^2, \quad \forall t\in [0, \delta).
		\end{equation}
	\end{enumerate}
\end{Lemma}
\begin{Remark}
	Let us note that \eqref{h1} implies that 
	\begin{equation}\label{h2}
		h \left( \frac{2}{(1-\tilde{\beta})^\frac{1}{2}} \bar{C}^{\frac{1}{2}}(T)\mathcal{E}^\frac{1}{2}(0) \right) <\frac{1-\tilde{\beta}}{2},
	\end{equation}
being the positive constant $C$ in \eqref{h1} bigger or equal than $\bar{C}(T)$ and being the function $h$ strictly increasing.
\end{Remark}
\begin{proof}
	From the assumption (H3) on the function $\psi$, we can write 
	\begin{equation}
		\label{assumptionPsi}
		\begin{array}{l}
			\displaystyle{|\psi(u)|\leq \int_0^1 |\langle \nabla \psi (su),u\rangle_H | ds} \\
			\hspace{1,15 cm}
			\displaystyle{\leq  ||A^\frac{1}{2}u||^2_H \int_0^1 h(s||A^\frac{1}{2}u||_H)sds}\\
			\hspace{1,15 cm}\displaystyle{\leq  h(||A^\frac{1}{2}u||_H)||A^\frac{1}{2}u||^2_H\int_{0}^{1}sds=\frac{1}{2}h(||A^\frac{1}{2}u||_H)||A^\frac{1}{2}u||^2_H,}
		\end{array}
	\end{equation}
where we used the fact that $h$ is a strictly increasing function and the fact that $||u||_{D(A^\frac{1}{2})}=(1-\tilde{\beta})||A^\frac{1}{2}u||_H$ with $\tilde{\beta}<1$.
\\Now, being $U$ a non-zero solution to \eqref{forma_astratta2}, the initial datum $\tilde{g}$ satisfies $\mathcal{B}\tilde{g}\neq 0$. Indeed, if the initial datum $\tilde{g}$ is such that $\mathcal{B}\tilde{g}\equiv 0$, then the unique solution to \eqref{forma_astratta2} is $U\equiv 0$. As a consequence $B^*g\neq 0$ since, otherwise, being $B$ a linear operator, we would have $0=BB^*g=\mathcal{B}\tilde{g}$. Hence, from the assumption $h (\Vert A^{\frac 12} u_0(0)\Vert_H) < \frac {1-\tilde \beta} 2$ and from \eqref{assumptionPsi}, we have that
	$$
	\begin{array}{l}
		\vspace{0.3cm}\displaystyle{ E(0)=\frac{1}{2}||u_1||_H^2+\frac{1-\tilde{\beta}}{2}||A^\frac{1}{2}u_0(0)||_H^2-\psi(u_0(0))+\frac{1}{2}\int_{-\bar\tau}^0 |k(s)|\cdot ||B^*g(s)||^2_H ds}\\
		\hspace{1.3 cm}
		\vspace{0.3cm}\displaystyle{+\frac{1}{2}\int_0^{+\infty} \beta(s)||A^\frac{1}{2}\eta_0(s)||_H^2 ds}\\
		\hspace{0,9 cm}
		\vspace{0.3cm}\displaystyle{\geq \frac{1}{2}||u_1||_H^2+\frac{1-\tilde{\beta}}{2}||A^\frac{1}{2}u_0(0)||_H^2-\frac{1}{2}h(||A^\frac{1}{2}u_0(0)||_H)||A^\frac{1}{2}u_0(0)||_H^2}\\
		\hspace{1.3 cm}
		\vspace{0.3cm}\displaystyle{+\frac{1}{2}\int_{-\bar\tau}^0 |k(s)|\cdot ||B^*g(s)||^2_H ds+\frac{1}{2}\int_0^{+\infty} \beta(s)||A^\frac{1}{2}\eta_0(s)||_H^2 ds}\\
		\hspace{0,9 cm}
		\vspace{0.3cm}\displaystyle{ \ge\frac{1}{2}||u_1||^2_H+\frac{1-\tilde{\beta}}{4}||A^\frac{1}{2}u_0(0)||^2_H +\frac{1}{2}\int_{-\bar\tau}^0 |k(s)| \cdot ||B^*g(s)||^2_H ds }\\
		\hspace{1.3 cm}
		\displaystyle{+\frac{1}{2}\int_0^{+\infty} \beta(s) ||A^\frac{1}{2} \eta_0(s)||^2_H ds>0.}
	\end{array}
	$$
	So, the claim $1$ is proven.
	\\In order to prove the second statement, we argue by contradiction. Let us denote
	$$
	r:=\sup \{ s\in [0,\delta) : \,\eqref{stima E dal basso}\, \text{holds},\, \forall t\in [0,s)\}.
	$$
	We suppose by contradiction that $r<\delta$. Then, by continuity, it holds
	\begin{equation}
		\label{continuita}
		\begin{array}{l}
			\displaystyle{E(r)=\frac{1}{4}||u_t(r)||^2_H+\frac{1-\tilde{\beta}}{4}||A^\frac{1}{2}u(r)||_H^2+\frac{1}{4}\int_{r-\bar\tau}^r |k(s)| \cdot ||B^*u_t(s)||_H^2 ds}\\
			\hspace{2 cm}
			\displaystyle{ +\frac{1}{4} \int_0^{+\infty} \beta(s)||A^\frac{1}{2}\eta^r(s)||_H^2 ds.}
		\end{array}
	\end{equation}
	In particular, \eqref{continuita} implies that
	$$
	\frac{1-\tilde{\beta}}{4} \Vert A^{\frac 1 2} u(r)\Vert^2_H\leq E(r).
	$$
	Also, by definition of $r$, for all $t\in [0,r]$, 

$$\begin{array}{l}
\displaystyle{E(t)\geq \frac{1}{4}||u_t(t)||_H^2+\frac{1-\tilde{\beta}}{4}||A^\frac{1}{2}u(t)||_H^2}\\
\displaystyle{\hspace{2 cm}+\frac{1}{4}\int_{t-\bar\tau}^t |k(s)| \cdot ||B^*u_t(s)||_H^2 ds +\frac{1}{4}\int_0^{+\infty} \beta(s) ||A^\frac{1}{2}\eta^t (s)||_H^2 ds}\\
\displaystyle{\hspace{1,5 cm}\geq \frac{1}{4}||u_t(t)||_H^2.}
\end{array}
$$
	Thus, the assumption \eqref{e(t)} of Lemma \ref{stimaE} is satisfied and we can write
	$$E(t)\leq \bar{C}(t)\mathcal{E}(0), \quad \forall t\in [0,r],$$
	from which, being $\bar{C}(t)\leq \bar{C}(T)$, it comes that
	$$E(t)\leq \bar{C}(T)\mathcal{E}(0),\quad \forall t\in [0,r].$$
	In particular, for $t=r$,
	$$E(r)\leq \bar{C}(T)\mathcal{E}(0).$$
	As a consequence, $$\frac{1-\tilde{\beta}}{4} \Vert A^{\frac 1 2} u(r)\Vert^2_H\leq E(r)\leq \bar{C}(T)\mathcal{E}(0).$$
	Thus, since $h$ is strictly increasing, from \eqref{h1} (which implies \eqref{h2}) we have that
	\begin{equation}\label{risultato}
		\begin{array}{l}
			\displaystyle{ h(||A^\frac{1}{2}u(r)||_H)\leq h\left( \frac{2}{(1-\tilde{\beta})^\frac{1}{2}}\bar{C}^\frac{1}{2}(T)\mathcal{E}^\frac{1}{2}(0)\right) <\frac{1-\tilde{\beta}}{2}.}
		\end{array}
	\end{equation}
	Finally, using \eqref{assumptionPsi} and \eqref{risultato} we can conclude that
	$$
	\begin{array}{l}
		\vspace{0.3cm}\displaystyle{ E(r)=
			\frac{1}{2}||u_t(r)||_H^2+\frac{1-\tilde{\beta}}{2}||A^\frac{1}{2}u(r)||_H^2-\psi(u(r))+\frac{1}{2}\int_{r-\bar\tau}^r|k(s)|\cdot ||B^*u_t(s)||^2_H ds}\\
		\hspace{3 cm}
		\displaystyle { +\frac{1}{2}\int_0^{+\infty} \beta(s) ||A^\frac{1}{2} \eta^r(s)||_H^2 ds}\\
		\hspace{1 cm} \displaystyle{
			>\frac{1}{4}||u_t(r)||_H^2+\frac{1-\tilde{\beta}}{4}||A^\frac{1}{2}u(r)||_H^2+\frac{1}{4}\int_{r-\bar\tau}^r|k(s)| \cdot ||B^*u_t(s)||_H^2 ds}\\
		\hspace{3 cm}
		\displaystyle{ +\frac{1}{4}\int_0^{+\infty} \beta(s) ||A^\frac{1}{2} \eta^r(s)||_H^2 ds.}
	\end{array}
	$$
	This contradicts the maximality of $r$. So, $r=\delta$ and the proof is completed.
\end{proof}

\section{Stability result}

\setcounter{equation}{0}

First, we give a stability result for the abstract model \eqref{forma_astratta2} under suitable assumptions. More precisely, we prove an exponential stability estimate for solutions to \eqref{forma_astratta2} corresponding to {\em small} initial data.

Our abstract result is the following.
\begin{Theorem}\label{StabilityAbstract}
Assume \eqref{ipotesi2}. Let $\tilde g\in C([-\bar\tau, 0];\mathcal{H})$ and let $U$ be a solution to \eqref{forma_astratta2} with the initial datum $\tilde{g}$, defined in a time interval $[0,T]$, $T>0$, that satisfies 
\begin{equation}\label{boundedsolution}
	\lVert U(t)\rVert_{\mathcal{H}}\leq C,\quad \forall t\in [0,T],
\end{equation}
for some $C>0$ such that $L(C)<\frac{\omega-\omega'}{M}$.
\\Then, $U$ satisfies the exponential decay estimate 
\begin{equation}
\label{stimaesponenziale}
\Vert  U(t)\Vert_{\mathcal H}\le Me^{\gamma}\left(||U_0||_{\mathcal H}+e^{\omega\bar{\tau}}Kb^2\max_{r\in[-\bar{\tau},0]}\{e^{\omega r}||\tilde{g}(r)||_{\mathcal H}\}\right)e^{-(\omega-\omega'-ML(C))t},
\end{equation}
for all $t\in [0,T]$.
\end{Theorem}

\begin{proof}
Let $\tilde{g}\in C([-\bar\tau, 0];\mathcal{H})$. Let $U$ be a solution to \eqref{forma_astratta2} with the initial datum $\tilde{g}$. 
From Duhamel's formula, for all $t\in [0,T]$ we have
$$U(t)=S(t)U_0+\int_0^tS(t-s)[-k(s){\mathcal B}U(s-\tau(s))+F(U(s))]ds.$$ 
Thus, using \eqref{semigruppo}, we get
$$\begin{array}{l}
	\vspace{0.3cm}\displaystyle{||U(t)||_{\mathcal H}\leq ||S(t)||_{\mathcal{L}(\mathcal{H})}||U_0||_{\mathcal H}+\int_0^t||S(t-s)||_{\mathcal{L}(\mathcal{H})}|k(s)|\cdot ||\mathcal{B}U(s-\tau(s))||_{\mathcal H}ds}\\
	\vspace{0.3cm}\displaystyle{\hspace{2.2cm}+\int_0^t||S(t-s)||_{\mathcal{L}(\mathcal{H})} ||F(U(s))||_{\mathcal H}ds}\\
	\vspace{0.3cm}\displaystyle{\hspace{1.6cm}\leq Me^{-\omega t}||U_0||_{\mathcal H}+Me^{-\omega t}\int_0^te^{\omega s}|k(s)|\cdot ||\mathcal{B}U(s-\tau(s))||_{\mathcal H}ds}\\
	\displaystyle{\hspace{2.2cm}+Me^{-\omega t}\int_0^te^{\omega s}||F(U(s))||_{\mathcal H}ds.}
\end{array}$$ 
Now, using the assumptions $(F_1)$ and $(F_2)$ on $F$ and taking into account of \eqref{boundedsolution}, we can write $$||F(U(s))||_{\mathcal H}=||F(U(s))-F(0)||_{\mathcal H}\leq L(C)||U(s)||_{\mathcal H}.$$
This last fact together with \eqref{tau_bounded} implies that
\begin{equation}\label{1}
\begin{array}{l}
\vspace{0.3cm}\displaystyle{ ||U(t)||_{\mathcal H}\leq Me^{-\omega t}||U_0||_{\mathcal H}+Me^{-\omega t} \int_0^t e^{\omega s} |k(s)|\cdot ||\mathcal{B}U(s-\tau(s))||_{\mathcal H}ds  } \\
\vspace{0.3cm}
\displaystyle{ \hspace{2.3 cm}+ML(C)e^{-\omega t} \int_0^t e^{\omega s} ||U(s)||_{\mathcal H} ds}\\
\vspace{0.3cm}\displaystyle{\hspace{1.6 cm} \leq Me^{-\omega t}||U_0||_{\mathcal H}+Me^{-\omega t} e^{\omega\bar{\tau}}\int_0^t e^{\omega (s-\tau(s))} |k(s)|\cdot ||\mathcal{B}U(s-\tau(s))||_{\mathcal H}ds }\\
\displaystyle{ \hspace{2.3 cm}+ML(C)e^{-\omega t} \int_0^t e^{\omega s} ||U(s)||_{\mathcal H} ds.}
\end{array}
\end{equation}
Now, we assume $t\geq\bar{\tau}$. We split
\begin{equation}\label{casotmagtau}
	\begin{array}{l}
		\vspace{0.3cm}\displaystyle{\int_0^t e^{\omega (s-\tau(s))} |k(s)|\cdot ||\mathcal{B}U(s-\tau(s))||_{\mathcal H}ds=\int_0^{\bar{\tau}} e^{\omega (s-\tau(s))} |k(s)|\cdot ||\mathcal{B}U(s-\tau(s))||_{\mathcal H}ds}\\
		\displaystyle{\hspace{2cm}+\int_{\bar{\tau}}^t e^{\omega (s-\tau(s))} |k(s)|\cdot ||\mathcal{B}U(s-\tau(s))||_{\mathcal H}ds.}
	\end{array} 
\end{equation}
We first estimate, using \eqref{K} with $t=\bar{\tau}$,
$$\begin{array}{l}
	\vspace{0.3cm}\displaystyle{\int_0^{\bar{\tau}} e^{\omega (s-\tau(s))} |k(s)|\cdot ||\mathcal{B}U(s-\tau(s))||_{\mathcal H}ds}\\
	\vspace{0.3cm}\displaystyle{\hspace{1cm}\leq b^2\int_0^{\bar{\tau}}  |k(s)|\left(\max_{r\in[-\bar{\tau},0]}\{e^{\omega r}||\tilde{g}(r)||_{\mathcal H}\}+\max_{r\in[0,s]}\{e^{\omega r}||U(r)||_{\mathcal H}\}\right) ds}\\
	\vspace{0.3cm}\displaystyle{\hspace{1cm}\leq Kb^2\max_{r\in[-\bar{\tau},0]}\{e^{\omega r}||\tilde{g}(r)||_{\mathcal H}\}+b^2\int_{0}^{\bar{\tau}}|k(s)|\max_{r\in[0,s]}\{e^{\omega r}||U(r)||_{\mathcal H}\} ds.}\\
	\displaystyle{\hspace{1cm}= Kb^2\max_{r\in[-\bar{\tau},0]}\{e^{\omega r}||\tilde{g}(r)||_{\mathcal H}\}+b^2\int_{0}^{\bar{\tau}}|k(s)|\max_{r\in[s-\bar{\tau},s]\cap[0,s]}\{e^{\omega r}||U(r)||_{\mathcal H}\} ds.}
\end{array}$$
Also,
$$\begin{array}{l}
	\vspace{0.3cm}\displaystyle{\int_{\bar{\tau}}^t e^{\omega (s-\tau(s))} |k(s)|\cdot ||\mathcal{B}U(s-\tau(s))||_{\mathcal H}ds\leq b^2\int_{\bar{\tau}}^t e^{\omega (s-\tau(s))} |k(s)|\cdot ||U(s-\tau(s))||_{\mathcal H}ds}\\
	\vspace{0.3cm}\displaystyle{\hspace{3cm}\leq b^2 \int_{\bar{\tau}}^t|k(s)|\max_{r\in[s-\bar{\tau},s]}\{e^{\omega r}||U(r)||_{\mathcal H}\} ds}\\
	\displaystyle{\hspace{3cm}= b^2\int_{\bar{\tau}}^t|k(s)|\max_{r\in[s-\bar{\tau},s]\cap[0,s]}\{e^{\omega r}||U(r)||_{\mathcal H}\} ds.}
\end{array}$$
Therefore, \eqref{casotmagtau} becomes 
\begin{equation}\label{casotmagtau2}
	\begin{array}{l}
		\vspace{0.3cm}\displaystyle{\int_0^t e^{\omega (s-\tau(s))} |k(s)|\cdot ||\mathcal{B}U(s-\tau(s))||_{\mathcal H}ds\leq Kb^2\max_{r\in[-\bar{\tau},0]}\{e^{\omega r}||\tilde{g}(r)||_{\mathcal H}\}}\\
		\displaystyle{\hspace{3cm}+\int_{0}^{t}b^2|k(s)|\max_{r\in[s-\bar{\tau},s]\cap[0,s]}\{e^{\omega r}||U(r)||_{\mathcal H}\} ds,}
	\end{array}
\end{equation}
for all $t\geq \bar{\tau}$.
\\On the other hand, if $t<\bar{\tau}$, using \eqref{K} it rather holds
$$\begin{array}{l}
	\vspace{0.3cm}\displaystyle{\int_0^{t} e^{\omega (s-\tau(s))} |k(s)|\cdot ||\mathcal{B}U(s-\tau(s))||_{\mathcal H}ds}\\
	\vspace{0.3cm}\displaystyle{\hspace{2cm}\leq b^2\int_0^t |k(s)| \max_{r\in[-\bar{\tau},0]}\{e^{\omega r}||\tilde{g}(r)||_{\mathcal H}\}ds+b^2\int_{0}^{t}|k(s)|\max_{r\in[0,s]}\{e^{\omega r}||U(r)||_{\mathcal H}\} ds}\\
	\vspace{0.3cm}
	\displaystyle{\hspace{2cm}= b^2\int_0^t |k(s)| \max_{r\in[-\bar{\tau},0]}\{e^{\omega r}||\tilde{g}(r)||_{\mathcal H}\}ds+b^2\int_{0}^{t}|k(s)|\max_{r\in[s-\bar{\tau},s]\cap[0,s]}\{e^{\omega r}||U(r)||_{\mathcal H}\} ds}\\
	\displaystyle{\hspace{2cm}\leq b^2\int_0^{\bar{\tau}} |k(s)| \max_{r\in[-\bar{\tau},0]}\{e^{\omega r}||\tilde{g}(r)||_{\mathcal H}\}ds+b^2\int_{0}^{t}|k(s)|\max_{r\in[s-\bar{\tau},s]\cap[0,s]}\{e^{\omega r}||U(r)||_{\mathcal H}\} ds}\\
	\displaystyle{\hspace{2cm}\leq Kb^2\max_{r\in[-\bar{\tau},0]}\{e^{\omega r}||\tilde{g}(r)||_{\mathcal H}\}+b^2\int_{0}^{t}|k(s)|\max_{r\in[s-\bar{\tau},s]\cap[0,s]}\{e^{\omega r}||U(r)||_{\mathcal H}\} ds.}
\end{array}$$
So, \eqref{casotmagtau2} holds for every $t\in [0,T]$. Putting \eqref{casotmagtau2} in \eqref{1}, we can write
$$\begin{array}{l}
	\vspace{0.3cm}\displaystyle{ ||U(t)||_{\mathcal H}\leq Me^{-\omega t}\left(||U_0||_{\mathcal H}+e^{\omega\bar{\tau}}Kb^2\max_{r\in[-\bar{\tau},0]}\{e^{\omega r}||\tilde{g}(r)||_{\mathcal H}\}\right)}\\
	\vspace{0.3cm}\displaystyle{\hspace{1cm}+Me^{-\omega t}e^{\omega\bar{\tau}}b^2\int_{0}^{t}|k(s)|\max_{r\in[s-\bar{\tau},s]\cap[0,s]}\{e^{\omega r}||U(r)||_{\mathcal H}\} ds+ML(C)e^{-\omega t} \int_0^t e^{\omega s} ||U(s)||_{\mathcal H} ds,}\\	
	\vspace{0.3cm}\displaystyle{ \hspace{0.5cm}\leq Me^{-\omega t}\left(||U_0||_{\mathcal H}+e^{\omega\bar{\tau}}Kb^2\max_{r\in[-\bar{\tau},0]}\{e^{\omega r}||\tilde{g}(r)||_{\mathcal H}\}\right)}\\
	\displaystyle{\hspace{1cm}+e^{-\omega t}\int_{0}^{t}(Me^{\omega\bar{\tau}}b^2|k(s)|+ML(C))\max_{r\in[s-\bar{\tau},s]\cap[0,s]}\{e^{\omega r}||U(r)||_{\mathcal H}\} ds,}
\end{array}$$
from which
$$
\begin{array}{l}
\displaystyle{ e^{\omega t} ||U(t)||_{\mathcal H} \leq  M\left(||U_0||_{\mathcal H}+e^{\omega\bar{\tau}}Kb^2\max_{r\in[-\bar{\tau},0]}\{e^{\omega r}||\tilde{g}(r)||_{\mathcal H}\}\right)}\\
\hspace{2 cm}
\displaystyle{\hspace{1cm} +\int_{0}^{t}(Me^{\omega\bar{\tau}}b^2|k(s)|+ML(C))\max_{r\in[s-\bar{\tau},s]\cap[0,s]}\{e^{\omega r}||U(r)||_{\mathcal H}\} ds,}
\end{array}
$$
for all $t\in [0,T]$. Thus,
$$
\begin{array}{l}
	\displaystyle{ \max_{r\in[t-\bar{\tau},t]\cap[0,t]}\{e^{\omega r} ||U(r)||_{\mathcal H} \}\leq  M\left(||U_0||_{\mathcal H}+e^{\omega\bar{\tau}}Kb^2\max_{r\in[-\bar{\tau},0]}\{e^{\omega r}||\tilde{g}(r)||_{\mathcal H}\}\right)}\\
	\hspace{2 cm}
	\displaystyle{\hspace{1.5cm} +\int_{0}^{t}(Me^{\omega\bar{\tau}}b^2|k(s)|+ML(C))\max_{r\in[s-\bar{\tau},s]\cap[0,s]}\{e^{\omega r}||U(r)||_{\mathcal H}\} ds.}
\end{array}
$$
Hence, denoted with 
$$\tilde{U}(t):=\max_{r\in[t-\bar{\tau},t]\cap[0,t]}\{e^{\omega r} ||U(r)||_{\mathcal H} \},$$
using Gronwall's inequality we get
$$
\Vert \tilde U(t)\Vert_{\mathcal H}\le
M\left(||U_0||_{\mathcal H}+e^{\omega\bar{\tau}}Kb^2\max_{r\in[-\bar{\tau},0]}\{e^{\omega r}||\tilde{g}(r)||_{\mathcal H}\}\right)e^{Mb^2e^{\omega\bar{\tau}}\int_{0}^{t}|k(s)|ds+ML(C)t}.
$$
Finally,
$$e^{\omega t}\Vert U(t)\Vert_{\mathcal H}\leq M\left(||U_0||_{\mathcal H}+e^{\omega\bar{\tau}}Kb^2\max_{r\in[-\bar{\tau},0]}\{e^{\omega r}||\tilde{g}(r)||_{\mathcal H}\}\right)e^{Mb^2e^{\omega\bar{\tau}}\int_{0}^{t}|k(s)|ds+ML(C)t},$$
which implies that
$$\Vert U(t)\Vert_{\mathcal H}\leq M\left(||U_0||_{\mathcal H}+e^{\omega\bar{\tau}}Kb^2\max_{r\in[-\bar{\tau},0]}\{e^{\omega r}||\tilde{g}(r)||_{\mathcal H}\}\right)e^{Mb^2e^{\omega\bar{\tau}}\int_{0}^{t}|k(s)|ds+ML(C)t}e^{-\omega t}.$$
From \eqref{ipotesi2}, we can conclude that
$$\Vert  U(t)\Vert_{\mathcal H}\le Me^{\gamma}\left(||U_0||_{\mathcal H}+e^{\omega\bar{\tau}}Kb^2\max_{r\in[-\bar{\tau},0]}\{e^{\omega r}||\tilde{g}(r)||_{\mathcal H}\}\right)e^{-(\omega-\omega'-ML(C))t},$$
for all $t\in [0,T]$, which proves the exponential decay estimate \eqref{stimaesponenziale}. 
\end{proof}
Now, we prove that, for sufficiently small initial data, the system \eqref{modello} has a unique global solution that satisfies \eqref{boundedsolution}, whenever the coefficient of the delay feedback $k(t)$ satisfies \eqref{ipotesi2}.
\begin{Theorem}\label{teorema2.5}
	Assume \eqref{ipotesi2}. There exist $\rho>0$ and $C_{\rho}>0$, with $L(C_{\rho})<\frac{\omega-\omega'}{2M}$, for which, if $\tilde g=(u_0(0),g,\eta_0)$ is such that
	\begin{equation}\label{smallness}
		\begin{array}{l}
			\displaystyle{||u_1||^2_{H}+(1-\tilde{\beta})||A^{\frac{1}{2}}u_0(0)||^2_{H}+\int_{-\bar\tau}^0|k(s)|\cdot||B^*g(s)||^2_{H}}\\
			\displaystyle{\hspace{5cm}+\int_{0}^{+\infty}\beta(s)||A^{\frac{1}{2}}\eta_0(s)||^2_{H}ds<\rho^2}
\end{array}
\end{equation}
and
\begin{equation}\label{smallness_bis}
			\displaystyle{\max_{s\in [-\bar{\tau},0]}||g(s)||_{H}<\rho,}
	\end{equation}
then the problem \eqref{modelloDafermos} with  initial datum $\tilde{g}$ has a unique global solution $U\in C([0,+\infty)),\mathcal{H})$ that satisfies
\begin{equation}\label{boundedsol2}
	\lVert U(t)\rVert_{\mathcal{H}}\leq C_{\rho},\quad \forall t\geq 0.
\end{equation} 
\end{Theorem}
\begin{proof}
	Let us fix a time $T$ sufficiently large, $T\ge \bar{\tau},$ such that
	\begin{equation}\label{CT}
		C_T:=4M^2e^{2\gamma}\max\left\{(1+Kb^2e^{\omega\bar{\tau}}),\, e^{\omega\bar{\tau}}\right\}\left(1+e^{2\omega\bar{\tau}}K^2b^4\right)e^{-(\omega-\omega')T}<1.
	\end{equation}
%%Note that, being $T$ large, we can assume, without loss of generality, that $T\geq %%\bar{\tau}$.
 Also, we set 
\begin{equation}\label{C*}
	C^*_T:=\sup\left\{e^{3b^2\int_{nT}^{(n+1)T}\lvert k(s)\rvert ds}: n\in \mathbb{N}\right\}.
\end{equation}
The assumption \eqref{K} ensures that  $C^*_T<+\infty.$
Note that  $C^*_T\ge \bar{C}(T),$ where $\bar{C}(T)$ is defined in \eqref{Cbar}.
Now, we pick $\rho>0$ in such a way that 
\begin{equation}\label{rho}
	\rho\leq \frac{(1-\tilde{\beta})^{\frac{1}{2}}}{2(C^*_T)^{\frac{1}{2}}}h^{-1}\left(\frac{1-\tilde{\beta}}{2}\right).
\end{equation} 
Let $(u_0(0),u_1,\eta_0)$ and $g$ be such that \eqref{smallness} and \eqref{smallness_bis}   holds true. Then, the initial datum $\tilde{g}$  satisfies the following smallness condition:
\begin{equation}\label{smallness2}
	\max_{s\in [-\bar{\tau},0]}||\tilde g(s)||_{\mathcal H}\leq \sqrt{2} \rho.
\end{equation}
Now, from Theorem \ref{lemma1senzabound} there exists a unique local solution $U(\cdot)$ to \eqref{forma_astratta2} with the initial datum $\tilde{g}(s)$, $s\in[-\bar{\tau},0]$ which is defined on a time interval $[0,\delta)$ and that satisfies the Duhamel's formula
\begin{equation}\label{primasol}
	U(t)=S(t)U_0+\int_{0}^{t}S(t-s)[-k(s)\mathcal{B}U(s-\tau(s))+F(U(s))]\,ds,
\end{equation} 
for all $t\in [0,\delta)$. Without loss of generality, we can suppose that $\delta\leq T$. Also, we can assume that $U$ is a non-zero solution to \eqref{forma_astratta2}, since, otherwise, \eqref{boundedsol2} is trivially satisfied. Thus, using \eqref{rho}, the assumption \eqref{smallness} on the initial data,
 and recalling that $h$ is a strictly increasing function, we get
$$h(||A^{\frac{1}{2}}u_0(0)||_{H})< h\left(\frac{\rho}{(1-\tilde{\beta})^{\frac{1}{2}}}\right)\leq h\left(\frac{1}{2(C_T^*)^{\frac{1}{2}}}h^{-1}\left(\frac{1-\tilde{\beta}}{2}\right)\right)\leq h\left(h^{-1}\left(\frac{1-\tilde{\beta}}{2}\right)\right)=\frac{1-\tilde{\beta}}{2},$$
were in the above inequality we used the fact that $2(C_T^*)^{\frac{1}{2}}\geq 1$. Hence, being $h(||A^{\frac{1}{2}}u_0(0)||_{H})< \frac{1-\tilde{\beta}}{2} $, from Lemma \ref{Lemma 2} it follows that $E(0)>0$. 
\\Also, using \eqref{assumptionPsi} and the fact that $h(||A^{\frac{1}{2}}u_0(0)||_{H})< \frac{1-\tilde{\beta}}{2} $, we can write
$$\begin{array}{l}
	\vspace{0.3cm}\displaystyle{E(0)=\frac{1}{2}||u_1||_H^2+\frac{1-\tilde{\beta}}{2}||A^\frac{1}{2}u_0(0)||_H^2-\psi(u_0(0))+\frac{1}{2}\int_{-\bar\tau}^0 |k(s)|\cdot ||B^*u_t(s)||^2_H ds}\\
	\vspace{0.3cm}\displaystyle{\hspace{2cm}+\frac{1}{2}\int_0^{+\infty} \beta(s)||A^\frac{1}{2}\eta_0(s)||_H^2 ds}\\
	\vspace{0.3cm}\displaystyle{\hspace{1cm}\leq\frac{1}{2}||u_1||_H^2+\frac{1-\tilde{\beta}}{2}||A^\frac{1}{2}u_0(0)||_H^2+\frac{1}{2}h(||A^\frac{1}{2}u_0(0)||_H)||A^\frac{1}{2}u_0(0)||^2_H}\\ \vspace{0.3cm}\displaystyle{\hspace{2cm}+\frac{1}{2}\int_{-\bar\tau}^0 |k(s)|\cdot ||B^*g(s)||^2_H ds+\frac{1}{2}\int_0^{+\infty} \beta(s)||A^\frac{1}{2}\eta_0(s)||_H^2 ds}\\
	\vspace{0.3cm}\displaystyle{\hspace{1cm}\leq\frac{1}{2}||u_1||_H^2+\frac{1-\tilde{\beta}}{2}||A^\frac{1}{2}u_0(0)||_H^2+\frac{1-\tilde{\beta}}{4}||A^\frac{1}{2}u_0(0)||^2_H}\\ \vspace{0.3cm}\displaystyle{\hspace{2cm}+\frac{1}{2}\int_{-\bar\tau}^0 |k(s)|\cdot ||B^*g(s)||^2_H ds+\frac{1}{2}\int_0^{+\infty} \beta(s)||A^\frac{1}{2}\eta_0(s)||_H^2 ds}\\
	\vspace{0.3cm}\displaystyle{\hspace{1cm}=\frac{1}{2}||u_1||_H^2+\frac{3}{4}(1-\tilde{\beta})||A^\frac{1}{2}u_0(0)||_H^2+\frac{1}{2}\int_{-\bar\tau}^0 |k(s)|\cdot ||B^*g(s)||^2_H ds}\\ \displaystyle{\hspace{2cm}+\frac{1}{2}\int_0^{+\infty} \beta(s)||A^\frac{1}{2}\eta_0(s)||_H^2 ds< \rho^2.}
\end{array} 
$$
The above inequality, together with \eqref{rho}, implies that
\begin{equation}\label{2}
	h\left(\frac{2}{(1-\tilde{\beta})^{\frac{1}{2}}}(C_T^*)^{\frac{1}{2}}E^{\frac{1}{2}}(0)\right)<h\left(\frac{2}{(1-\tilde{\beta})^{\frac{1}{2}}}(C_T^*)^{\frac{1}{2}}\rho\right)\leq h\left(h^{-1}\left(\frac{1-\tilde\beta}{2}\right)\right)=\frac{1-\tilde\beta}{2}.
\end{equation}
Furthermore, from \eqref{smallness_bis}, 
\begin{equation}\label{3}
	h\left(\frac{2}{(1-\tilde{\beta})^{\frac{1}{2}}}(C_T^*)^{\frac{1}{2}}\cdot\frac{1}{\sqrt{2}}\max_{s\in [-\bar{\tau},0]}||g(s)||_{H}\right)< h\left(\frac{\sqrt{2}}{(1-\tilde{\beta})^{\frac{1}{2}}}(C_T^*)^{\frac{1}{2}}\rho\right)\le \frac{1-\tilde\beta}{2}.
\end{equation}
Then, from \eqref{2}, \eqref{3} and by definition of $\mathcal{E}(0)$, we have that
\begin{equation}\label{4}
h\left(\frac{2}{(1-\tilde{\beta})^{\frac{1}{2}}}(C_T^*)^{\frac{1}{2}}\mathcal{E}^{\frac{1}{2}}(0)\right)<\frac{1-\tilde{\beta}}{2}.
\end{equation}
Since $C_T^*\geq \bar{C}(T)$, the above inequality \eqref{4} allows us to apply Lemma \ref{Lemma 2}, that ensures that \eqref{stima E dal basso} is satisfied for all $t\in [0,\delta)$. In particular,
$$E(t)\geq \frac{1}{4}||u_t(t)||_{H}^2,\quad \forall t\in [0,\delta).$$
Combining \eqref{disuguaglianza energia} and \eqref{stima E dal basso} with $\bar{C}(t)\leq \bar{C}(T)\leq C_T^*$, it turns out that
\begin{equation}\label{5}
	\begin{array}{l}
		\vspace{0.3cm}\displaystyle{\frac{1}{4}||u_t(t)||_H^2+\frac{1-\tilde{\beta}}{4}||A^\frac{1}{2}u(t)||_H^2+\frac{1}{4}\int_{t-\bar\tau}^t |k(s)| \cdot ||B^*u_t(s)||_H^2 ds
			+\frac{1}{4}\int_0^{+\infty} \beta(s) ||A^\frac{1}{2}\eta^t (s)||_H^2 ds}\\
		\displaystyle{\hspace{6cm}<E(t)\leq C_T^*\mathcal E(0),}
\end{array}
 \end{equation}
for all $t\in [0,\delta)$. Thus, since the solution $U$ is bounded from \eqref{5}, we can extend it in $t=\delta$ and in the whole interval $[0,T]$. Moreover, \eqref{5} holds for all $t\in [0,T]$.
\\Now, using \eqref{4} and \eqref{5} with $t=T$, we can write
\begin{equation}\label{t=T}
	h(||A^{\frac{1}{2}}u_0(T)||_{H})\leq h\left(\frac{2}{(1-\tilde{\beta})^{\frac{1}{2}}}(C_T^*)^{\frac{1}{2}}\mathcal{E}^{\frac{1}{2}}(0)\right)<\frac{1-\tilde{\beta}}{2}.
\end{equation}
Furthermore, from the smallness assumption \eqref{smallness} on the initial data and from \eqref{5}, it comes that
$$\frac{1}{4}||U(t)||^2_{\mathcal H}\leq E(t)\leq C_T^*\rho^2,$$
where here we have used the fact that ${\mathcal E}(0)<\rho^2$. Thus, 
\begin{equation}\label{boundedU}
	||U(t)||_{\mathcal H} \leq C_{\rho}:=2(C_T^*)^{\frac{1}{2}}\rho.
\end{equation}
Next, eventually choosing a smaller value of $\rho$, we can suppose that $L(C_\rho)<\frac{\omega-\omega'}{2M}$. We have so proved that there exist $\rho>0$, $C_\rho>0$ such that, whenever the initial data $(u_0(0),g,\eta_0)$ satisfy the smallness condition \eqref{smallness} and \eqref{smallness_bis}, then the system \eqref{forma_astratta2} with the initial data $U_0$, $\tilde{g}$, has a unique solution $U$ defined in the time interval $[0,T]$ such that $||U(t)||_{\mathcal H} \leq C_{\rho}$. 
\\As a consequence, from Theorem \ref{StabilityAbstract}, being $L(C_{\rho})<\frac{\omega-\omega'}{2M}$, the following estimate holds
$$\Vert  U(t)\Vert_{\mathcal H}\le Me^{\gamma}\left(||U_0||_{\mathcal H}+e^{\omega\bar{\tau}}Kb^2\max_{r\in[-\bar{\tau},0]}\{e^{\omega r}||\tilde{g}(r)||_{\mathcal H}\}\right)e^{-\frac{\omega-\omega'}{2}t},$$
for all $t\in [0,T]$. Thus, using \eqref{smallness2}, we can write
$$\Vert  U(t)\Vert_{\mathcal H}\le Me^{\gamma}\left(||U_0||_{\mathcal H}+e^{\omega\bar{\tau}}Kb^2 \sqrt{2} \rho\right)e^{-\frac{\omega-\omega'}{2}t},$$
for any $t\in [0,T]$. Then, 
$$\begin{array}{l}
	\vspace{0.3cm}\displaystyle{\Vert  U(t)\Vert^2_{\mathcal H}\le M^2e^{2\gamma}\left(||U_0||_{\mathcal H}+e^{\omega\bar{\tau}}Kb^2 \sqrt{2} \rho\right)^2e^{-(\omega-\omega')t}}\\
	\vspace{0.3cm}\displaystyle{\hspace{1.5cm}\leq2M^2e^{2\gamma}\left(||U_0||^2_{\mathcal H}+(e^{\omega\bar{\tau}}Kb^2 \sqrt{2} \rho )^2\right)e^{-(\omega-\omega')t}}\\
	\displaystyle{\hspace{1.5cm}=2M^2e^{2\gamma}\left(||U_0||^2_{\mathcal H}+e^{2\omega\bar{\tau}}K^2b^4 2\rho^2\right)e^{-(\omega-\omega')t},}
\end{array}$$
from which, taking into account \eqref{smallness2} with $U_0=\tilde{g}(0)$,
\begin{equation}\label{6}
	\Vert  U(t)\Vert^2_{\mathcal H}\le 4M^2e^{2\gamma}\rho^2\left(1+e^{2\omega\bar{\tau}}K^2b^4\right)e^{-(\omega-\omega')t},
\end{equation}
for all $t\in [0,T]$.
\\Also, for all $s\in [T-\bar{\tau},T]\subseteq [0,T]$, \eqref{6} yields
$$\begin{array}{l}
	\vspace{0.3cm}\displaystyle{||u_t(s)||_H^2\leq \Vert  U(s)\Vert^2_{\mathcal H}}\\
	\vspace{0.3cm}\displaystyle{\hspace{2.1cm}\leq 4M^2e^{2\gamma}\rho^2\left(1+e^{2\omega\bar{\tau}}K^2b^4\right)e^{-(\omega-\omega')s}}\\
	\vspace{0.3cm}\displaystyle{\hspace{2.1cm}\leq 4M^2e^{2\gamma}\rho^2\left(1+e^{2\omega\bar{\tau}}K^2b^4\right)e^{-(\omega-\omega')(T-\bar{\tau})}}\\
	\displaystyle{\hspace{2.1cm}\leq 4M^2e^{2\gamma}e^{\omega\bar{\tau}}\rho^2\left(1+e^{2\omega\bar{\tau}}K^2b^4\right)e^{-(\omega-\omega')T}.}
\end{array}$$
As a consequence, from \eqref{K}
$$\begin{array}{l}
	\vspace{0.3cm}\displaystyle{\int_{T-\bar{\tau}}^{T}|k(s)| \cdot ||B^*u_t(s)||_H^2 ds\leq 4M^2e^{2\gamma}e^{\omega\bar{\tau}}b^2\rho^2\left(1+e^{2\omega\bar{\tau}}K^2b^4\right)e^{-(\omega-\omega')T}\int_{T-\bar{\tau}}^{T}|k(s)|ds}\\
	\displaystyle{\hspace{3cm}\leq 4KM^2e^{2\gamma}e^{\omega\bar{\tau}}b^2\rho^2 \left(1+e^{2\omega\bar{\tau}}K^2b^4\right)e^{-(\omega-\omega')T}.}
\end{array}$$
This last fact together with \eqref{6} implies that
 \begin{equation}\label{7}
	\begin{array}{l}
		\vspace{0.3cm}\displaystyle{\Vert  U(T)\Vert^2_{\mathcal H}+\int_{T-\bar{\tau}}^{T}|k(s)| \cdot ||B^*u_t(s)||_H^2 ds}\\
		\displaystyle{\hspace{1cm}\leq 4M^2e^{2\gamma}\rho^2 \left(1+e^{2\omega\bar{\tau}}K^2b^4\right)\left(1+Kb^2e^{\omega\bar{\tau}}\right)e^{-(\omega-\omega')T}\leq C_T\rho^2<\rho^2.}
	\end{array}
\end{equation}
Moreover, \eqref{6} implies that
$$\max_{s\in [T-\bar{\tau},T]}||u_t(s)||_H^2\leq 4 M^2e^{2\gamma}e^{\omega\bar{\tau}}\rho^2 \left(1+e^{2\omega\bar{\tau}}K^2b^4\right)e^{-(\omega-\omega')T}\leq C_T\rho^2<\rho^2,$$
from which 
\begin{equation}\label{8}
	\max_{s\in [T-\bar{\tau},T]}||u_t(s)||_H<\rho.
\end{equation}
The conditions \eqref{7} and \eqref{8} allow us to apply the same arguments employed before on the interval $[T,2T]$. Namely, we consider  the initial value problem
\begin{equation}\label{modello2}
	\begin{array}{l}
		\displaystyle{V'(t)=\mathcal{A}V(t)-k(t)\mathcal{B}V(t-\tau(t))+F(V(t)),\quad t\in [T,2T],}\\
		\displaystyle{V(s)=U(s),\quad s\in [T-\bar{\tau},T],}
	\end{array}
\end{equation}
where $U(\cdot)$ is the solution to \eqref{forma_astratta2} in the interval $[0,T]$.
\\We define now the energy of the solution
\begin{equation}\label{energy}
	\begin{array}{l}
		\vspace{0.3cm}\displaystyle{\tilde{E}(t):=\tilde{E} (v(t))=\frac{1}{2}||v_t(t)||_H^2+\frac{1-\tilde{\beta}}{2}||A^{\frac 12}v(t)||^2_H-\psi(v)  }\\ \hspace{2 cm}
		\displaystyle{ +\frac{1}{2}\int_0^{+\infty} \beta(s) ||A^{\frac 1 2} \eta^t(s)||^2_H ds +\frac{1}{2}\int_{t-\bar\tau}^t |k(s)|\cdot ||B^*v_t(s)||_H^2 ds,}
	\end{array}
\end{equation}
where $\eta^t(s)=v(t)-v(t-s)$, and the functional
\begin{equation}\label{funzionale}
	\tilde{\mathcal{E}}(t):=\max\left\{\frac{1}{2}\max_{s\in [T-\bar{\tau},T]}\Vert u_t(s)\Vert_H^2,\,\,\max_{s\in[T,t]}\tilde{E}(s)\right\}.
\end{equation}
Let us note that $\tilde{E}(T)=E(T)$. 
\\Now, from Theorem \ref{lemma1senzabound}, the problem \eqref{modello2} with the initial datum $U(s)$, $s\in [T-\bar\tau,T]$, has a unique local solution $V(\cdot)$ defined on a time interval $[T,T+\delta)$ given by the Duhamel's formula
\begin{equation}\label{secondasol}
	V(t)=S(t-T)U(T)+\int_{0}^{t}S(t-T-s)[-k(s)\mathcal{B}V(s-\tau(s))+F(V(s))]\, ds,
\end{equation} 
for all $t\in [T,T+\delta)$. We can assume that $\delta\leq \bar{\tau}$ and that $V$ is a non-zero solution, since, otherwise, \eqref{boundedsol2} is obviously satisfied for all $t\geq T$.
\\First of all, the inequality \eqref{t=T} yields $\tilde{E}(T)>0$.
\\Let us note that, if $\tilde{E}(t)\geq \frac{1}{4}\lVert v_t(t)\rVert^2_{H}$, for all $t\in [T,T+\delta)$, arguing as in Lemma \ref{stimaE},
\begin{equation}\label{stimaEmodificata}
	\tilde{E}(t)\leq e^{3b^2\int_{T}^{t}\lvert k(s)\rvert ds}\,\tilde{\mathcal{E}}(T),\quad \forall t\in [T,T+\delta).
\end{equation}
Also, using the same arguments employed in Lemma \ref{Lemma 2} and observing that
$$C^*_T\ge e^{3b^2\int_T^{2T}\vert k(s)\vert ds },$$ 
if
\begin{equation}\label{h3}
	h\left(\frac{2}{(1-\tilde{\beta})^{\frac{1}{2}}}(C^*_T)^{\frac{1}{2}}\tilde{\mathcal{E}}^{\frac{1}{2}}(T)\right)<\frac{1-\tilde{\beta}}{2},
\end{equation}
 then 
\begin{equation}\label{stima E dal bassomodificata}
	\begin{array}{l}
		\vspace{0.3cm}\displaystyle{ \tilde{E}(t)>\frac{1}{4}||v_t(t)||_H^2+\frac{1-\tilde{\beta}}{4}||A^\frac{1}{2}v(t)||_H^2+\frac{1}{4}\int_{t-\bar\tau}^t |k(s)| \cdot ||B^*v_t(s)||_H^2 ds}\\
		\hspace{1.5 cm}
		\displaystyle{ +\frac{1}{4}\int_0^{+\infty} \beta(s) ||A^\frac{1}{2}\eta^t (s)||_H^2 ds,}
	\end{array}
\end{equation}
for all $t\in[T,T+ \delta)$. In particular,
\begin{equation}\label{J2modificata}
\tilde{E}(t)> \frac 14 \Vert V(t)\Vert_{\mathcal H}^2, \quad \forall t\in [T, T+\delta).
\end{equation}
Now, from \eqref{7}, we have $\tilde{E}(T)<\rho^2$. This together with \eqref{8} implies that
\begin{equation}\label{E<rho}
	\tilde{\mathcal{E}}(T)<\rho^2.
\end{equation}
Hence, using \eqref{rho} we get
 \begin{equation}\label{9}
	h\left(\frac{2}{(1-\tilde{\beta})^{\frac{1}{2}}}(C^*_T)^{\frac{1}{2}}\tilde{\mathcal{E}}(T)^{\frac{1}{2}}\right)<h\left(\frac{2}{(1-\tilde{\beta})^{\frac{1}{2}}}(C^*_T)^{\frac{1}{2}}\rho\right)<\frac{1-\tilde{\beta}}{2}.
\end{equation}
So, \eqref{h3} is satisfied. As a consequence, inequalities \eqref{stima E dal bassomodificata} and \eqref{J2modificata} hold for all $t\in [T,T+\delta)$. In particular, from \eqref{J2modificata} we get $$\tilde E(t)\geq \frac{1}{4}\lVert v_t(t)\rVert_H^2,\quad \forall t\in [T,T+\delta).$$
Therefore, also inequality \eqref{stimaEmodificata} is fulfilled. Combining \eqref{stimaEmodificata} and \eqref{stima E dal bassomodificata}, we finally get
\begin{equation}\label{solb}
	\begin{array}{l}
		\vspace{0.2cm}\displaystyle{\frac{1}{4}||v_t(t)||_H^2+\frac{1-\tilde{\beta}}{4}||A^\frac{1}{2}v(t)||_H^2+\frac{1}{4}\int_{t-\bar\tau}^t |k(s)| \cdot ||B^*v_t(s)||_H^2 ds+\frac{1}{4}\int_0^{+\infty} \beta(s) ||A^\frac{1}{2}\eta^t (s)||_H^2 ds}\\
		\displaystyle{\hspace{3cm}<\tilde{E}(t)\leq e^{3b^2\int_{T}^{2T}\lvert k(s)\rvert ds}\tilde {\mathcal{E}}(T)\leq C^*_T\tilde{\mathcal{E}}(T).}
	\end{array}
\end{equation}
Thus, the solution $V$ is bounded and we can extend it in $t=T+\delta$ and in the whole interval $[T,2T]$. Moreover, \eqref{solb} and \eqref{J2modificata} hold true for all $t\in [T,2T]$. So, taking into account of \eqref{E<rho}, it holds
\begin{equation}\label{boundedV}
	\lVert V(t)\rVert_{\mathcal{H}}\leq 2(C_T^*)^{\frac{1}{2}}\tilde{\mathcal{E}}^{\frac{1}{2}}(T)\leq 2(C_T^*)^{\frac{1}{2}}\rho=C_\rho.
\end{equation}
Putting together the two partial solutions \eqref{primasol} and \eqref{secondasol} to \eqref{forma_astratta2} obtained in the time intervals $[0,T]$ and $[T,2T]$, respectively, we get the existence of a unique solution $U\in C([0,2T];\mathcal{H})$ to \eqref{forma_astratta2} that is defined in the time interval $[0,2T]$ and that satisfies the Duhamel's formula \eqref{primasol}, for all $t\in [0,2T]$. Moreover, from \eqref{boundedU} and \eqref{boundedV}, the solution $U$ satisfies \eqref{boundedsolution} with $C=C_\rho$. Thus, since $L(C_\rho)<\frac{\omega-\omega'}{2M}$, the exponential decay estimate \eqref{stimaesponenziale} is satisfied by the solution $U$, i.e. $$\Vert  U(t)\Vert_{\mathcal H}\le Me^{\gamma}\left(||U_0||_{\mathcal H}+e^{\omega\bar{\tau}}Kb^2\max_{r\in[-\bar{\tau},0]}\{e^{\omega r}||\tilde{g}(r)||_{\mathcal H}\}\right)e^{-\frac{\omega-\omega'}{2}t},\quad\forall t\in [0,2T].$$
Again, we deduce that \eqref{6} holds for all $t\in [0,2T]$. Furthermore, for all $s\in [2T-\bar{\tau},2T],$
$$\begin{array}{l}
	\vspace{0.3cm}\displaystyle{||u_t(s)||_H^2 \leq ||U(s)||_{\mathcal{H}}^2\leq  4M^2e^{2\gamma}e^{\omega\bar{\tau}}\rho^2\left(1+e^{2\omega\bar{\tau}}K^2b^4\right)e^{-(\omega-\omega')2T}.}
\end{array}$$
As a consequence,
\begin{equation}\label{10}
\begin{array}{l}
		\vspace{0.3cm}\displaystyle{\Vert  U(2T)\Vert^2_{\mathcal H}+\int_{2T-\bar{\tau}}^{2T}|k(s)| \cdot ||B^*u_t(s)||_H^2 ds}\\
		\vspace{0.3cm}\displaystyle{\hspace{1cm}\leq 4 M^2e^{2\gamma}\rho^2(1+e^{2\omega \bar{\tau}}K^2b^4)e^{-(\omega-\omega')2T}}\\
		\vspace{0.3cm}\displaystyle{\hspace{2cm}+4 M^2e^{2\gamma}e^{\omega\bar{\tau}}Kb^2\rho^2 (1+e^{2\omega \bar{\tau}}K^2b^4)e^{-(\omega-\omega')2T}}\\
		\vspace{0.3cm}\displaystyle{\hspace{1cm}\leq 4 M^2e^{2\gamma}\rho^2 (1+e^{2\omega \bar{\tau}}K^2b^4)(1+Kb^2e^{\omega\bar{\tau}})e^{-(\omega-\omega')2T}}\\
		\vspace{0.3cm}\displaystyle{\hspace{1cm}\leq 4 M^2e^{2\gamma}\rho^2 (1+e^{2\omega \bar{\tau}}K^2b^4)(1+Kb^2e^{\omega\bar{\tau}})e^{-(\omega-\omega')T}}\\
		\displaystyle{\hspace{1cm}\leq C_T\rho^2<\rho^2.}
	\end{array}
\end{equation}
Also,
$$\begin{array}{l}
	\vspace{0.3cm}\displaystyle{\max_{s\in [2T-\bar{\tau},2T]}||u_t(s)||_H^2\leq 4M^2e^{2\gamma}e^{\omega\bar{\tau}}\rho^2 \left(1+e^{2\omega\bar{\tau}}K^2b^4\right)e^{-(\omega-\omega')2T}}\\
	\displaystyle{\hspace{2cm}\leq 4M^2e^{2\gamma}e^{\omega\bar{\tau}}\rho^2 \left(1+e^{2\omega\bar{\tau}}K^2b^4\right)e^{-(\omega-\omega')T}\leq C_T\rho^2<\rho^2,}
\end{array}$$
from which 
\begin{equation}\label{12}
	\max_{s\in [2T-\bar{\tau},2T]}||u_t(s)||_H<\rho.
\end{equation}
The conditions \eqref{10} and \eqref{12} allows us to apply the same arguments employed before on the interval $[2T,3T]$, obtaining the existence of a unique solution to \eqref{forma_astratta2} defined in the time interval $[0,3T]$ and that satisfies \eqref{boundedsol2} for all $t\in [0,3T]$. Iterating this procedure, we finally get the existence of a unique global solution $U\in C([0,+\infty);\mathcal{H})$ to \eqref{forma_astratta2} with the initial datum $\tilde{g}(s)$, $s\in [-\bar{\tau},0]$, that satisfies \eqref{boundedsol2}.
\end{proof}	
We have then proved that, for suitably small initial data, solutions to problem  \eqref{forma_astratta2} are globally defined and bounded by a positive constant $C_{\rho}$ that satisfies $L_{C_{\rho}}<\frac{\omega-\omega'}{M}$ (see Theorem \ref{teorema2.5}). Thus, solutions to problem  \eqref{forma_astratta2} satisfy the exponential decay estimate \eqref{stimaesponenziale}. Therefore, we are ready to prove  the energy decay
for model \eqref{modelloDafermos}.

\begin{Theorem}\label{teorema_finale}
Let us consider model \eqref{modelloDafermos} and assume \eqref{ipotesi2}. There exists $\rho>0$ such that,
if the following smallness conditions on  the initial data are satisfied:
\begin{equation}\label{ipotesi}
\begin{array}{l}
	\displaystyle{||u_1||^2_{H}+(1-\tilde{\beta})||A^{\frac{1}{2}}u_0(0)||^2_{H}+\int_{-\bar\tau}^0|k(s)|\cdot||B^*g(s)||^2_{H}}\\
	\displaystyle{\hspace{4cm}+\int_{0}^{+\infty}\beta(s)||A^{\frac{1}{2}}\eta_0(s)||^2_{H}ds<\rho^2,}
\end{array}
\end{equation}
and
\begin{equation}\label{ipotesi_bis}
\displaystyle{\max_{s\in [-\bar{\tau},0]}||g(s)||_{H}<\rho,}
\end{equation}
then \eqref{modelloDafermos} has a unique solution  globally defined. Moreover, the energy satisfies the exponential decay estimate
\begin{equation}
\label{tesi}
E(t)\leq \tilde{C}e^{-\mu t},
\end{equation}
where $\mu:=\omega-\omega'$ and $\tilde C$ is a constant depending on the initial data.
\end{Theorem}

\begin{proof}
Let $\rho>0$ be the positive constant in Theorem \ref{teorema2.5}. Let us consider initial data for which the smallness conditions \eqref{ipotesi} and   \eqref{ipotesi_bis} hold true. Then, from Theorem \ref{teorema2.5}, the problem \eqref{forma_astratta2} has a unique global $U(\cdot)$ solution that satisfies $||U(t)||_\mathcal{H}<C_{\rho}$, with $L(C_{\rho})<\frac{\omega-\omega'}{2M},$ and the exponential decay estimate \eqref{stimaesponenziale}. From \eqref{assumptionPsi} and \eqref{5}, we have that
$$
\begin{array}{l}
	\vspace{0.3cm}\displaystyle{E(t)=\frac{1}{2}||u_t(t)||_H^2+\frac{1-\tilde{\beta}}{2}||A^{\frac 12}u(t)||^2_H-\psi(u(t))  +\frac{1}{2}\int_0^{+\infty} \beta(s) ||A^{\frac 1 2} \eta^t(s)||^2_H ds}\\
	\vspace{0.3cm}\displaystyle{\hspace{1.5cm}+\frac{1}{2}\int_{t-\bar\tau}^t |k(s)|\cdot ||B^*u_t(s)||_H^2 ds}\\
	\vspace{0.3cm}\displaystyle{\hspace{1cm}\leq\frac{1}{2}||u_t(t)||_H^2+\frac{1-\tilde{\beta}}{2}||A^{\frac 12}u(t)||^2_H+\frac{1-\tilde{\beta}}{4}||A^{\frac 12}u(t)||^2_H +\frac{1}{2}\int_0^{+\infty} \beta(s) ||A^{\frac 1 2} \eta^t(s)||^2_H ds}\\
	\vspace{0.3cm}\displaystyle{\hspace{1.5cm}+\frac{b^2}{2}\int_{t-\bar\tau}^t |k(s)|\cdot ||u_t(s)||_H^2 ds}\\
	\vspace{0.3cm}\displaystyle{\hspace{1cm}\leq||U(t)||_\mathcal{H}^2+\frac{b^2}{2}\int_{t-\bar\tau}^t |k(s)|\cdot ||U(s)||_{\mathcal{H}}^2 ds}
\end{array} 
$$
for any $t\geq \bar{\tau}$. Now, applying Theorem \ref{StabilityAbstract}, there exists a positive constant $\hat{C}$ such that
\begin{equation}\label{11}
||U(t)||_\mathcal{H} \leq \hat{C}e^{-\frac{\omega-\omega'}{2} t},\quad\forall t\geq 0,
\end{equation}
where here we have used the fact that $L(C_{\rho})<\frac{\omega-\omega'}{2M}$.
Thus,
$$E(t)\leq \hat{C}^2
e^{-(\omega-\omega') t}+\frac{\hat{C}^2b^2}{2}\int_{t-\bar{\tau}}^{t}|k(s)|e^{-(\omega-\omega')s}ds\leq \hat{C}^2e^{-(\omega-\omega') t}+\frac{\hat{C}^2b^2}{2}Ke^{\omega\bar{\tau}}e^{-(\omega-\omega')t},$$for all $t\geq \bar{\tau}$. Setting
$$\tilde{C}:=\max\left\{\hat{C}^2,\frac{\hat{C}^2b^2}{2}Ke^{\omega\bar{\tau}}\right\},$$
we can write
$$E(t)\leq \tilde{C}e^{-(\omega-\omega')t},$$
from which, 
$$E(t)\leq \tilde{C}e^{-\mu t}, \quad\forall t\geq \bar{\tau},$$
where $\mu =\omega-\omega'$. Hence, \eqref{tesi} holds true for all $t\geq \bar{\tau}$.
\end{proof}

\section{Examples}

\setcounter{equation}{0}

\subsection{The wave equation with memory and source term}\label{ex1}
Let $\Omega$ be a non-empty bounded set in $\RR^n$, with boundary $\Gamma$ of class $C^2,$ and  let $\mathcal{O}\subset \Omega$ be a nonempty open subset of $\Omega$. We assume $n\ge 3.$ The lower dimension cases could be studied analogously. We consider the following wave equation:
\begin{equation}
\label{memory+source}
\begin{array}{l}
\displaystyle{u_{tt}(x,t)-\Delta u(x,t)+\int_0^{+\infty} \beta(s)\Delta u(x,t-s) ds+k(t)\chi _{\mathcal{O}} u_t(x,t-\tau(t))}\\
\hspace{7 cm}
\displaystyle{=|u(x,t)|^\sigma u(x,t), \qquad \text{in} \ \Omega \times (0,+\infty),}\\
\displaystyle{ u(x,t)=0, \qquad \text{in} \ \Gamma \times (0,+\infty),}\\
\displaystyle{ u(x,t)=u^0(x,t) \qquad \text{in} \ \Omega \times (-\infty,0],}\\
\displaystyle{ u_t(x,0)=u_1(x), \qquad \text{in} \ \Omega,} \\
\displaystyle{ u_t(x,t)=g(x,t), \qquad \text{in} \ \Omega \times [-\bar\tau,0],}
\end{array}
\end{equation}
where the time delay function $\tau (\cdot)$ satisfies \eqref{tau_bounded}, $\beta :(0,+\infty)\rightarrow (0,+\infty)$ is a locally absolutely continuous memory kernel satisfying the assumptions (i)-(iv), $\sigma >0$ and the damping coefficient $k(\cdot)$ is a function in $L^1_{loc}([-\bar{\tau},+\infty))$ for which \eqref{K} holds.
Then, system \eqref{memory+source} falls in the form \eqref{modello} with $A=-\Delta$ and $D(A)=H^2(\Omega)\cap H^1_0(\Omega).$ Moreover, $D(A^{\frac{1}{2}})=H^1_0(\Omega).$ 
\\Let us define $\eta_s^t$ as in \eqref{eta}. Then, system \eqref{memory+source} can be rewritten as follows:
\begin{equation}
\label{memory+source_riscritta}
\begin{array}{l}
\displaystyle{ u_{tt}(x,t)-(1-\tilde{\beta})\Delta u(x,t)-\int_0^{+\infty} \beta(s)\Delta \eta^t(x,s) ds+k(t)\chi _{\mathcal{O}} u_t(x,t-\tau)}\\
\hspace{7 cm}
\displaystyle{=|u(x,t)|^\sigma u(x,t), \quad \text{in} \ \Omega\times (0,+\infty),}\\
\displaystyle{\eta_t^t(x,s)=-\eta_s^t(x,s)+u_t(x,t), \qquad \text{in} \ \Omega\times (0,+\infty)\times (0,+\infty),}\\
\displaystyle{ u(x,t)=0, \qquad \text{in} \ \Gamma \times (0,+\infty),}\\
\displaystyle{ \eta^t(x,s)=0, \qquad \text{in} \ \Gamma \times (0,+\infty), \quad \text{for} \ t\geq 0,}\\
\displaystyle{ u(x,0)=u_0(x):=u^0(x,0), \qquad \text{in} \ \Omega,}\\
\displaystyle{ u_t(x,0)=u_1(x):=\frac{\partial u^0}{\partial t}(x,t)\Bigr| _{t=0}, \qquad \text{in} \ \Omega,}\\
\displaystyle{ \eta^0(x,s)=\eta_0(x,s):=u^0(x,0)-u^0(x,-s), \qquad \text{in} \ \Omega\times (0,+\infty),}\\
\displaystyle{ u_t(x,t)=g(x,t), \qquad \text{in} \ \Omega\times [-\bar\tau,0].}

\end{array}
\end{equation}
In order to reformulate \eqref{memory+source_riscritta} as an abstract first order equation, we introduce the Hilbert space $L^2_\beta ((0,+\infty);H^1_0(\Omega))$ endowed with the inner product
$$
\langle \phi , \varphi\rangle _{L^2_\beta ((0,+\infty);H^1_0(\Omega))} := \int_{\Omega} \left( \int_0^{+\infty} \beta(s) \nabla \phi (x,s) \nabla \varphi (x,s) ds \right) dx,
$$
and the Hilbert space
$$
\mathcal{H}=H_0^1(\Omega)\times L^2(\Omega) \times L^2_\beta ((0,+\infty);H^1_0(\Omega)),
$$
equipped with the inner product
\begin{equation*}
\left\langle
\left (
\begin{array}{l}
u\\
v\\
w
\end{array}
\right ),
\left (
\begin{array}{l}
\tilde u\\
\tilde v\\
\tilde w
\end{array}
\right )
\right\rangle_{\mathcal{H}}:=(1-\tilde{\beta}) \int_{\Omega} \nabla u \nabla \tilde u dx +\int_{\Omega} v\tilde v dx+ \int_{\Omega} \int_0^{+\infty} \beta(s) \nabla w \nabla \tilde w ds dx.
\end{equation*}
We set $U=(u,u_t,\eta^t)$. Then, \eqref{memory+source_riscritta} can be rewritten in the form \eqref{forma_astratta2}, where
$$
\mathcal{A} \begin{pmatrix}
u\\
v\\
w
\end{pmatrix}
=
\begin{pmatrix}
v\\
(1-\tilde{\beta})\Delta u+\int_0^{+\infty} \beta(s) \Delta w(s) ds \\
-w_s +v
\end{pmatrix},
$$
with domain
\begin{eqnarray*}
{D(\mathcal{A})}&=&\{  (u,v,w)\in H_0^1(\Omega)\times H_0^1(\Omega)\times L^2_\beta ((0,+\infty);H^1_0(\Omega)): \\
& & (1-\tilde{\beta})u +\int_0^{+\infty} \beta(s)w(s)ds \in H^2(\Omega)\cap H_0^1(\Omega), \ w_s\in L^2_\beta ((0,+\infty);H^1_0(\Omega))\} ,
\end{eqnarray*}
$\mathcal{B}(u,v,\eta^t)^T := (0,\chi _{\mathcal{O}} v,0)^T$, $\tilde{g}=(u_0,g,\eta^0)^T$, in $[-\bar{\tau},0]$, and $F(U(t))=(0,|u(t)|^{\sigma}u(t), 0)^T$. 
\\Now, we consider the functional 
$$
\psi (u):= \frac{1}{\sigma +2} \int_{\Omega} |u(x)|^{\sigma+2} dx, \quad \forall u\in H^1_0(\Omega).
$$
By Sobolev's embedding theorem, $\psi$ is well-defined for $0<\sigma\leq\frac{4}{n-2}$. Also, the G\^ateaux derivative of $\psi$ at any point $u\in H_0^1(\Omega)$ is
$$
D\psi (u)(v)=\int_{\Omega} |u(x)|^\sigma u(x)v(x)dx,
$$
for all $v\in H_0^1(\Omega)$. Moreover, as in \cite{ACS}, if $0<\sigma\le \frac{2}{n-2}$, then $\psi$ satisfies the assumptions (H1), (H2), (H3). 
\\Let us define the energy in this way:
$$
\begin{array}{l}
\displaystyle{E(t):= \frac{1}{2}\int_{\Omega} |u_t(x,t)|^2 dx +\frac{1-\tilde{\beta}}{2}\int_{\Omega} |\nabla u(x,t)|^2 dx -\psi(u(x,t))}\\
\hspace{2 cm}
\displaystyle{ +\frac{1}{2}\int_{t-\bar\tau}^t \int_{\mathcal{O}} |k(s)|\cdot |u_t(x,s)|^2 dx ds +\frac{1}{2}\int_0^{+\infty} \beta (s)\int_{\Omega} |\nabla \eta^t(x,s)|^2 dx ds.}
\end{array}
$$
Then, from Theorem \ref{teorema_finale} we have that \eqref{memory+source} is well-posed and that, for solutions corresponding to suitably small initial data, an exponential decay estimate holds provided that the condition \eqref{ipotesi2} is satisfied.

\subsection{The wave equation with memory and integral source term}

Let $\Omega$ be a non-empty bounded set in $\RR^n$, with boundary $\Gamma$ of class $C^2$. Moreover, let $\mathcal{O}\subset \Omega$ be a nonempty open subset of $\Omega$. We consider the following wave equation:
\begin{equation}
	\label{memory+integral}
	\begin{array}{l}
		\displaystyle{u_{tt}(x,t)-\Delta u(x,t)+\int_0^{+\infty} \beta(s)\Delta u(x,t-s) ds+k(t)\chi _{\mathcal{O}} u_t(x,t-\tau(t))}\\
		\hspace{7 cm}
		\displaystyle{=\left(\int_{\Omega}|u(x,t)|^2\right)^\frac{p}{2} u(x,t), \qquad \text{in} \ \Omega \times (0,+\infty),}\\
		\displaystyle{ u(x,t)=0, \qquad \text{in} \ \Gamma \times (0,+\infty),}\\
		\displaystyle{ u(x,t)=u_0(x,t) \qquad \text{in} \ \Omega \times (-\infty,0],}\\
		\displaystyle{ u_t(x,0)=u_1(x), \qquad \text{in} \ \Omega,} \\
		\displaystyle{ u_t(x,t)=g(x,t), \qquad \text{in} \ \Omega \times [-\bar\tau,0],}
	\end{array}
\end{equation}
where the time delay function $\tau (\cdot)$ satisfies \eqref{tau_bounded}, $\beta :(0,+\infty)\rightarrow (0,+\infty)$ is a locally absolutely continuous memory kernel such that the assumptions (i)-(iv) are fulfilled, $p\geq 1,$ and the damping coefficient $k(\cdot)$ is a function in $L^1_{loc}([-\bar{\tau},+\infty))$ for which \eqref{K} holds true.
Then, system \eqref{memory+integral} falls in the form \eqref{modello} with $A=-\Delta$, $D(A)=H^2(\Omega)\cap H^1_0(\Omega)$ and $D(A^{\frac{1}{2}})= H^1_0(\Omega).$ 
\\Let us note that system \eqref{memory+integral} is analogous to system \eqref{memory+source}, with the only difference given by the nonlinearity. So, arguing as in Example \ref{ex1}, \eqref{memory+integral} can be rewritten as \eqref{memory+source_riscritta} and then as an abstract first order equation.
\\Now, consider the functional
$$
\psi (u):= \frac{1}{p +2} \left(\int_{\Omega} |u(x)|^{2}dx\right)^{\frac{p+2}{2}} =\frac{1}{p +2}\lVert u\rVert_{L^2(\Omega)}^{p+2},\quad \forall u\in L^2(\Omega).
$$
Note that $\psi$ is well defined. Also, the G\^ateaux derivative of $\psi$ at any point $u\in L^2(\Omega)$ is given by
$$D\psi(u)(v)= \left(\int_{\Omega} |u(x)|^{2}dx\right)^{\frac{p}{2}}\int_{\Omega}u(x)v(x)dx,$$
for any $v\in L^2(\Omega)$. Then, $\psi$ is defined in the whole $L^2(\Omega)$ and 
$$\nabla\psi(u)=\left(\int_{\Omega} |u(x)|^{2}dx\right)^{\frac{p}{2}}u(x),\quad \forall u\in L^2(\Omega),$$
is the unique vector representing $D\psi(u)$ in the Riesz isomorphism. So $(H_1)$ is trivially satisfied.
Arguing as in \cite{LuoXiao}, we can find a positive constant $C>0$ such that
\begin{equation}\label{ass2}
	\lVert \nabla\psi(u)-\nabla\psi(v)\rVert^2_{L^2(\Omega)}\leq C(\lVert u\rVert^{2p}_{H^1_0(\Omega)}+\lVert v\rVert^{2p}_{H^1_0(\Omega)})\lVert u-v\rVert^2_{H^1_0(\Omega)},
\end{equation}
for all $u,v\in H^1_0(\Omega).$ Thus, for any $r> 0$ and for all $u,v\in H^1_0(\Omega)$ with $\lVert \nabla u\rVert_{L^2(\Omega)},\lVert \nabla v\rVert_{L^2(\Omega)}\leq r$, since from Poincaré inequality $\lVert\cdot\rVert_{H^1_0(\Omega)}$ and $\lVert \nabla (\cdot)\rVert_{L^2(\Omega)}$ are equivalent norms on $H^1_0(\Omega)$, from \eqref{ass2} we get
$$\lVert \nabla\psi(u)-\nabla\psi(v)\rVert^2_{L^2(\Omega)}\leq 2r^{2p}C\lVert \nabla u-\nabla v\rVert^2_{L^2(\Omega)},$$
from which
$$\lVert \nabla\psi(u)-\nabla\psi(v)\rVert_{L^2(\Omega)}\leq \sqrt{2C}r^{p}\lVert \nabla u-\nabla v\rVert_{L^2(\Omega)},$$
Hence, $(H_2)$ is satisfied. 
\\Finally, we prove that $(H_3)$ holds true. Note that $\psi(0),\nabla\psi(0)=0$. Also, using \eqref{ass2} with $v=0$ and Poincaré inequality, for all $u\in H^1_0(\Omega)$ we can write
$$\lVert \nabla\psi(u)\rVert^2_{L^2(\Omega)}\leq C\lVert \nabla u\rVert^{2p}_{L^2(\Omega)}\lVert \nabla u\rVert^2_{L^2(\Omega)},$$
which implies
$$\lVert \nabla\psi(u)\rVert_{L^2(\Omega)}\leq \sqrt{C}\lVert \nabla u\rVert^{p}_{L^2(\Omega)}\lVert \nabla u\rVert_{L^2(\Omega)}$$
Thus, $(H_3)$ is fulfilled with $h(z)=\sqrt{C}z^{p}$, for all $z\geq 0$, which is a continuous and strictly increasing function.
\\Let us define the energy as follows:
$$
\begin{array}{l}
	\displaystyle{E(t):= \frac{1}{2}\int_{\Omega} |u_t(x,t)|^2 dx +\frac{1-\tilde{\beta}}{2}\int_{\Omega} |\nabla u(x,t)|^2 dx -\psi(u(x,t))}\\
	\hspace{2 cm}
	\displaystyle{ +\frac{1}{2}\int_{t-\tau}^t \int_{\mathcal{O}} |k(s)|\cdot |u_t(x,s)|^2 dx ds +\frac{1}{2}\int_0^{+\infty} \beta (s)\int_{\Omega} |\nabla \eta^t(x,s)|^2 dx ds.}
\end{array}
$$
Then, applying Theorem \ref{teorema_finale} to this model, we get well-posedness and exponential decay of the energy for solutions corresponding to suitably small initial data provided that the condition \eqref{ipotesi2} is satisfied.

\bigskip
\noindent {\bf Acknowledgements.} The authors are members of  {\it Gruppo Nazionale per l'Analisi Ma\-te\-matica, la Probabilit\`a e le loro Applicazioni (GNAMPA)} of the Istituto Nazionale di Alta Matematica (INdAM).
They  are partially supported by PRIN 2022  (2022238YY5) {\it Optimal control problems: analysis,
approximation and applications} and by INdAM GNAMPA Project {\it ``Modelli alle derivate parziali per interazioni multiagente non 
simmetriche"}(CUP E53C23001670001). 
C. Pignotti is also partially supported by
PRIN-PNRR 2022 (P20225SP98) {\it Some mathematical approaches to climate change and its impacts}.

\end{document}